\documentclass[11pt]{article}

\usepackage{microtype}
\usepackage{amsmath,amsthm,amssymb,fullpage}
\usepackage[all]{xy}
\usepackage{xspace}
\usepackage{graphicx}
\usepackage{wrapfig}
\usepackage[all]{xy}

\mathchardef\mhyphen="2D
\newcommand{\id}{\mathrm{id}}
\newcommand{\dist}{\mathrm{d}}
\newcommand{\R}{\mathbb{R}}
\newcommand{\e}{\varepsilon}

\newcommand{\nerve}{\mathrm{Nrv\;}}

\newcommand{\im}{\mathrm{im\xspace}}

\newcommand{\dgm}{\mathrm{D}}
\newcommand{\CW}{\mathrm{CW}}

\newcommand{\hocolim}{\mathrm{hocolim}\;}

\newcommand{\B}{{\mathcal{B}}}

\newcommand{\U}{{\mathcal{U}}}

\newcommand{\W}{\mathcal{W}}
\newcommand{\N}{\mathcal{N}}

\newcommand{\rank}{\mathrm{rk\xspace}}

\newtheorem{theorem}{Theorem}
\newtheorem{lemma}[theorem]{Lemma}
\newtheorem{definition}[theorem]{Definition}
\newtheorem{corollary}{Corollary}[theorem]
\newtheorem{proposition}[theorem]{Proposition}

\title{The Generalized Persistent Nerve Theorem}

  \author{Nicholas J. Cavanna\\
University of Connecticut\\
  \texttt{nicholas.j.cavanna@gmail.com}
 \and 
 Donald R.~Sheehy\\
 University of Connecticut\\
 \texttt{don.r.sheehy@gmail.com}
 }
\date{}

\begin{document}
\maketitle

\begin{abstract}
 The Nerve Theorem equates the homotopy type of a suitably covered topological space with that of a combinatorial simplicial complex called a nerve.
After filtering a space one can compute the filtration's persistent homology.
In persistence theory  the Nerve Theorem has the Persistent Nerve Lemma as an analogue which equates the persistent homology of a filtration of spaces and that of the filtration of nerves corresponding to a filtration of covers assuming each cover is good.
As nerves are discrete, their geometric realizations can serve as proxies for topological spaces and filtrations in algorithmic settings, e.g.~the \v Cech and Rips complexes are nerves so one can compute their homology over various scales rather than that of a well-sampled space.

In this paper we introduce a parameterized generalization of a good cover filtration called an $\e$-good cover.
It is defined as a cover filtration in which the reduced homology groups of the image of the inclusions between the intersections of the cover filtration at two scales $\e$ apart are trivial.
Assuming that one has an $\e$-good cover filtration of a finite simplicial filtration, we prove a tight bound on the bottleneck distance between the persistence diagrams of the nerve filtration and the simplicial filtration that is linear with respect to $\e$  and the homology dimension in question.
Quantitative guarantees for covers that are not good are useful for when one has a cover of a non-convex metric space, or one has more generally constructed simplicial covers that are not the result of triangulations of metric balls.
These guarantees also aid in situations when constructing a good cover filtration is computationally intensive, but a smaller $\e$-good cover's construction is feasible.

Other notable contributions  are the introduction of an interleaving between the nerve and covered space's chain complexes up to chain homotopy and the constructive nature of the interleaving that ultimately provides the bound on the bottleneck distance,
The Persistent Nerve Lemma is a direct corollary of our main theorem as good covers are $0$-good covers.
 Furthermore, we  symmetrize the asymmetric interleaving used to prove the bound by shifting the nerve filtration, improving the interleaving distance by a factor of $2$.
\end{abstract}
\pagebreak

\section{Introduction}
\label{sec:introduction}

  The Nerve Theorem~\cite{borsuk48imbedding} is an important link between topological spaces and discrete geometric and topological algorithms.
  It is at the heart, either implicitly or explicitly, of many foundational algorithms in the rapidly growing field of topological data analysis.
  Example problems include surface reconstruction~\cite{chazal08towards,edelsbrunner94triangulating,edelsbrunner95union},
  function reconstruction~\cite{chazal11scalar},
  homology inference~\cite{chazal06topology,niyogi08finding,niyogi11topological}
  , coordinate-free sensor network coverage~\cite{cavanna17when,desilva07coverage}, shape analysis~\cite{chazal09gromov-hausdorff},
  data modeling~\cite{dey16multiscale,singh07topological,stovner12mapper}, and clustering~\cite{chazal13persistence,edelsbrunner17topological}.

  A \emph{cover} of a simplicial complex $W$ is a collection of subcomplexes $\U = \{U_0,\ldots, U_n\}$  such that $W\subseteq \bigcup_{i=0}^n U_i$. \footnote{The notation $W$ is intended to help the reader remember that $W$, (\emph{``double U''}) is the union of the $U$'s.}
  The \emph{nerve} of $\U$ is the abstract simplicial complex defined as follows
  \[
    \nerve\U := \{\textrm{non-empty }\sigma \subseteq [n] \mid \bigcap_{i\in \sigma} U_i\neq \emptyset \}.
  \]
  A cover $\U$ is \emph{good} if for every $\sigma\subseteq [n]$ the simplicial complex $\bigcap_{i\in \sigma}U_i$ is empty or contractible.
  The Nerve Theorem equates the homotopy type and thus homology of the covered simplicial complex $W$ to that of the nerve of a good cover $\U$.
  This theorem allows one to construct algorithms that compute topological properties of the nerve of a particular good cover and relate the output back to infer properties of the covered space.

  In persistence theory, one often works with filtered topological spaces: $\W = (W^\alpha)_{\alpha\ge 0}$, where $W^\alpha\subseteq W^\beta$ for all $\alpha \le \beta$.
  The Persistent Nerve Lemma of Chazal and Oudot~\cite{chazal08towards} proves that the Nerve Theorem extends in the most natural way to the persistence setting.
  In particular, it implies that the \emph{union filtration} $\W = (\bigcup_{i=0}^n U_i^\alpha)_{\alpha\ge 0}$ and the \emph{nerve filtration} $\nerve\U = (\nerve \U^\alpha)_{\alpha\ge 0}$ have the same persistent homology assuming $\U^\alpha := \{U_0^\alpha,\ldots, U_n^\alpha\}$ is a good cover of $W^\alpha$ for all $\alpha$.

The requirement of good covers to ensure topological theoretical guarantees has  significant algorithmic implications as it significantly reduces the spaces one can work in.
Many algorithms in topological data analysis utilize a standard pipeline where one considers a nice sample of some space, builds a simplicial complex from this sample called the \v Cech complex, which is the nerve of metric balls, and uses the fact that this has the same topology of the metric balls.
  This means that algorithms depending on the Nerve Theorem can only be applied to
spaces that admit covers by convex sets in Euclidean spaces or smooth manifolds of sufficiently large convexity radius~\cite{chazal11scalar} can be considered.
  Even one small hole in an intersection of cover elements can render nerve-based computational algorithms invalid as the theory rests upon some interpretation of this theorem.
  This also has implications for triangulations of covers of surfaces that have marginal measurement errors, because the errors can cause the cover elements to no longer be convex for example.
  Nerves are also in coverage testing for homological sensor networks~\cite{desilva07coverage,cavanna17when}, however the idealized model of Euclidean balls as coverage regions differs significantly from the very jagged coverage regions measured in practice, particularly when taking into account the affect of physical obstacles on real-life sensors' detection ranges.
  
 In this paper we introduce a parametrization of the good cover condition  for simplicial cover filtrations called an $\e$-good cover, which roughly says that the homology of the cover elements' intersections only persists for some $\e>0$ amount of time.
 Our main result is as follows.
 Given a simplicial cover filtration that is an $\e$-good cover filtration of the corresponding covered simplicial filtration, there exists a constructive $(K+1)\e$-interleaving between the $K$-dimensional persistence modules of the a finite simplicial filtration and its covers' nerve filtration, which implies a tight bound on the bottleneck distance between each modules persistence diagram.
 We assume no structure on the simplicial cover filtration other than that the simplicial complexes are finite.
 This persistence module interleaving notably results from a ``pseudo" interleaving between chain complexes where we only require the maps compositions are chain homotopic to the identity chain maps between scales.
 Due to the spaces in question being finite simplicial complexes, the chain maps used to construct the homological interleaving are computable.
 
 A corollary of the Generalized Persistent Nerve Theorem is the algebraic Persistent Nerve Lemma for simplicial filtrations and covers.
We also reduce the interleaving distance and thus the bottleneck distance bound by a factor of $2$ in Section~\ref{sec:biased} by considering a time-scale shifted nerve filtration assuming one knows an upper bound on the $\e$-goodness of the cover.
  
\paragraph*{History of the Problem}
In August $2016$, Govc and Skraba posted the solution to a very similar problem, among other results in their paper \emph{An Approximate Nerve Theorem} to the arXiv.
They originally assumed that the reduced homology of the $k$-wise intersections of a simplicial cover filtrations' cover elements are $\e$-interleaved with the $0$ module, a condition they call an $\e$-acyclic cover.
Their paper arrived at an identical bound, up to differences in definitions.
They restricted their hypotheses to a filtered simplicial complex that is induced by a cover on the complete complex.
The cover filtration at each scale was thus defined by the simplicial filtration at that scale's intersection with the cover.
Unfortunately, this assumption is too restrictive to imply a simplicial version of the Persistent Nerve Lemma which was a primary goal of this research for us, as it does not account for covers which have no inherent relation to each other outside of inclusions.

In September $2016$, we submitted our results for presentation at the 26th Fall Workshop on Computational Geometry, which notably implied the Persistent Nerve Lemma as a corollary, for our space assumptions, as originally desired.
Soon afterwards, Govc and Skraba updated their arXiv submission relaxing their cover filtration assumption, and their paper has recently been accepted to a journal~\cite{govc17approximate}.

Though our papers both prove a similar result, the approaches are very different  and utilize and develop different tools.
Govc and Skraba utilized a construction from homological algebra called a spectral sequence, as well as their novel right and left persistence module interleavings to prove their theorem, which is actually a direct result of the the module interleavings they compute between the pages of the spectral sequences and the persistence modules of the nerve and space filtrations.

In contrast, our proof technique focuses on constructing maps between the chain complexes of the nerve filtration, space filtration, and the barycentric decomposition of the so-called blow-up complex.
We go from the homological $\e$-goodness condition to corresponding chain maps and chain homotopies between them among the chain complexes of the spaces of interest.
This is more inline with the approach traditionally used to proof the Nerve Theorem and the homotopy version of the Persistent Nerve Lemma.
Due to the fact that these chain maps are defined on chain complexes of simplicial complexes and regular CW-complexes, they are computable in practice.
Our module interleaving and bottleneck distance bound results are a consequence of a chain-theoretic generalization of an interleaving only up to chain homotopy, rather than being purely one concerning persistent homology.
There are also the novel contributions of the creation of chain map between the barycentric decomposition of the nerve of a cover and the space filtration at a further time scale, and the use a technique we call \emph{lifting} to form a chain map into the barycentric decomposition of the blowup complex.
The lack of existence of such a map is the reason the Nerve Theorem fails when the cover is not good, which gives credence to the notion that the maps we construct are ``natural" choices.

\paragraph*{Related Work}

Apart from the Persistent Nerve Lemma and it's original homotopy version, researchers have examined related problems concerning covers and their associated nerves with respect to persistence theory.

Botnan and Spreemann~\cite{botnan15approximating} proved that if three cover filtrations are $\delta$-interleaved, where two of the filtrations are good and sandwich the third in the interleaving, then the bottleneck distance between the persistence module of the nerve of one of the good cover filtrations and the arbitrary one is upper bounded by $\delta$.

Dey et al.~\cite{dey17topological} prove that for a cover whose elements are path-connected, the $1$-dimensional homology of the map from the covered space to the nerve of the cover is surjective.
Using this result, they also prove that if there exists a so-called cover map between two covers, then the $1$-dimensional homology of the simplicial map resulting from the cover map between the covers' nerves is surjective.

\section{Background}\label{sec:background}

This is an overview of the combinatorial, topological and algebraic structures used in the paper to prove the Generalized Persistent Nerve Theorem.
See Hatcher's Algebraic Topology~\cite{Hatcher01} for further reference on chain complexes and homology, and Chazal et al.'s The Structure and Stability of Persistence Modules~\cite{chazal16structure} for more on persistence theory.

\subsection{Simplicial Complexes}
A \textbf{geometric simplex} is the convex closure of a set of affinely independent points.
A \textbf{(geometric) simplicial complex} $X$ is a collection of simplices such that for each simplex in $X$, each of its subsimplices are in $X$, and the intersection of two simplices is in $X$ or is empty.
An \textbf{abstract simplicial complex} $X$  over a finite vertex set $[n]:=\{0,\ldots, n\}$ is a subset of the powerset $2^{[n]}$ closed under taking subsets.
For a simplex $\sigma\subseteq [n]$, its dimension is $\#\{\sigma\}-1$.
A simplex of dimension $k$ is called a $k$-simplex of $X$.
The dimension of $X$ is the dimension of its largest simplex.

Each finite abstract simplicial complex has a corresponding geometric simplicial complex.
Consider the function $f: [n]\rightarrow \R^{n+1}$, where  $[n] \ni i\mapsto f(i)= e_i=(0,\ldots, 1,\ldots, 0)\in \R^{n+1}$.
For each maximal $\sigma$ of an abstract simplicial complex $X$, its geometric realization is $
|\sigma|:= \textrm{conv}\{f(i)\mid i\in\sigma\}$, i.e. the convex closure of the image of its vertices under $f$.
The \textbf{geometric realization} of $X$ is $|X|:= \bigcup_{\sigma\in X} |\sigma|$ which is equipped with the subspace topology inherited from the Euclidean topology.
This construction allows for discussion of topological properties of abstract simplicial complexes and is functorial in the sense that given two abstract simplicial complexes with a simplicial map between them, the geometric realization carries the simplicial map to a continuous structure-preserving map between the complexes' realizations.

CW-complexes are topogical structures that generalize the gluing procedures used to construct simplicial complexes.
A \textbf{CW-complex} $X$ is defined inductively, starting with a collection $X_0$ of $0$-cells, vertices, and then for natural $k$, $X_k$ is the union of $X_{k-1}$ and some $k$-cells whose boundaries are glued via continuous maps called attaching maps to the $(k-1)$-cells of $X_{k-1}$.
A CW-complex $X$ is called finite if $X=X_n$ for some $n$.

As an example, the $n$-sphere can be viewed as a CW-complex with one $0$-cell and one $n$-cell where the $n$-cell's boundary is attached to the $0$-cell via the constant map.
A simplicial complex is  CW-complex under the obvious gluing procedure.
Given two finite CW-complexes $X$ and $Y$, $X\times Y$ is a finite CW-complex with cells of the form $e_X\times e_Y$, where $e_X\in X$ and $e_Y\in Y$, and $\dim(e_X\times e_Y)=\dim(e_X)+\dim(e_Y)$.

  \begin{figure}[h]
    \centering
    \includegraphics[width=0.3\textwidth]{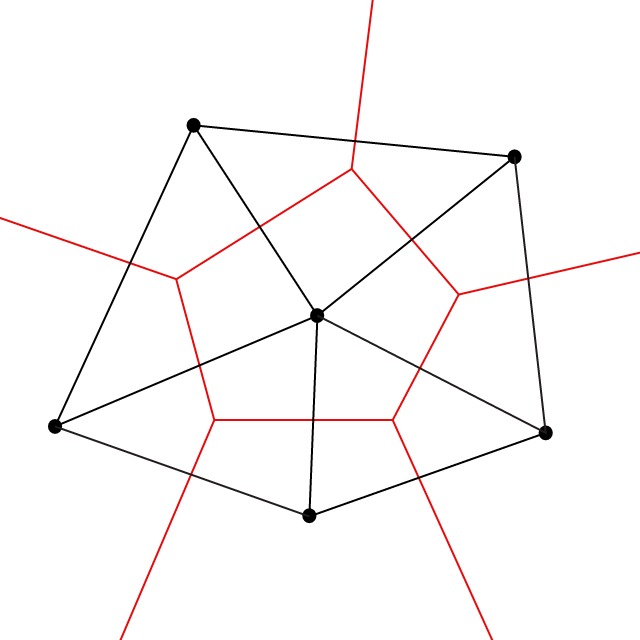}
    \caption{The Delaunay triangulation is a realization of the nerve of the Voronoi cells.}
    \label{fig:eps_good_cover}
  \end{figure}

 A \textbf{closed cover} of a simplicial complex $X$ is a collection of simplicial complexes $\mathcal{U} = \{U_a \}_{a\in A}$ all defined over the same vertex set  such that $X \subseteq \bigcup_{a\in A} U_a$ and $X$ is a subcomplex of of $\bigcup_{a\in A} U_a$.
 A space is \textbf{contractible} if it has the homotopy type of a point.
 For example, all $n$-simplices are contractible.
A cover where every nonempty intersections of finitely many elements of $\mathcal{U}$ is is called a \textbf{good cover}.
Given a closed cover $\mathcal{U}=\{U_a\}_{a\in A}$ indexed over indexign set $A$, the \textbf{nerve} of $\mathcal{U}$ is the abstract simplicial complex 
\[
\nerve\mathcal{U}=\{\sigma\subseteq A \mid \bigcap_{i\in\sigma} U_i\neq\emptyset \textrm{ and }\#\{\sigma\}<\infty\}. 
\]
The following is a version of the Nerve Theorem which relates the homotopy type of the nerve of a  cover to that of the covered space.
See Corollary 4G.3 in \cite{Hatcher01} for the more general topological formulation.

\begin{theorem}[Nerve Theorem]\label{thm:nerve_thm}
If $\mathcal{U}$ is a closed cover of a finite simplicial complex $X$ such that every non-empty intersection of finitely many of the elements in $\mathcal{U}$ is contractible, then $X$ is homotopy equivalent to the geometric realization of the nerve, $|\nerve\mathcal{U}|$.
\end{theorem}

\subsection{Chain Complexes and Homology}\label{subsec:homology}

We restrict ourselves to chain complexes and homology groups over $\mathbb{Z}/2\mathbb{Z}$ to simplify boundary computations, but the constructions hold for general finite fields as well.
Given a simplicial complex $X$, and a non-negative integer $k$, a \textbf{simplicial $k$-chain} over $\mathbb{Z}/2\mathbb{Z}$ is a formal sum of the form $\sum_i c_i\sigma_i$, where each $c_i\in \mathbb{Z}/2\mathbb{Z}$ and each $\sigma_i$ is a unique $k$-simplex of $X$.
Collectively these simplicial $k$-chains form a vector space/abelian group over $\mathbb{Z}/2\mathbb{Z}$ called the $k$-dimensional \textbf{chain group}, denoted by $C_k(X)$, which has a natural basis consisting of the $k$-simplices of $X$.
Formally, $C_{-1}(X)=0$.
For each $k$ there exists a linear map $\partial_k: C_k\rightarrow C_{k-1}$, called the \textbf{simplicial boundary map}, defined on a $k$-simplex $\sigma= [v_0,\ldots, v_k]$ by $\partial_k(\sigma):= \sum_{i=0}^k [v_0,\ldots, \hat{v_i}, \ldots, v_k]$, where $[v_0,\ldots, \hat{v_i}, \ldots, v_k ]$ is the $i$-th face of $\sigma$ or alternatively the simplex spanned by the vertices of $\sigma$ with $v_i$ removed.
The boundary map extends linearly to arbitrary $k$-chains and has the property that $\partial_k\partial_{k+1}=0$, or $\partial^2=0$ for short.
We denote the sequence of simplicial chain groups $(C_k(X))_{k\geq 0}$ with the appropriate boundary maps as $C_*(X)$ and call it the \textbf{simplicial chain complex} of $X$.

Given a CW-complex $X$, we can similarly define the \textbf{cellular chain complex} $C^{\CW}_*(X):= (C_k^{\CW}(X))_{k\geq 0}$, where each $C_k^{\CW}(X)$ is naturally isomorphic to the vector space over over $\mathbb{Z}/2\mathbb{Z}$ with a basis being the $k$-cells in $X$ .
For each basis $k$-cell $e^k_i$ in $C_k^{\CW}(X)$, there is a boundary map $\partial_k(e^k_i):=\sum_j d_{i,j} e^{k-1}_j$, where $d_{i,j}$ is computed by the cellular boundary formula (see page 140 in Hatcher\cite{Hatcher01} for the exact formula).
For a so-called regular CW-complex, which simplicial complexes and their products are, all the coefficients $d_{i,j}=1$.

Given two simplicial complexes $X$ and $Y$, a \textbf{chain map} $f:C_*(X)\rightarrow C_*(Y)$ is a collection of maps $f=(f_k)_{k\geq 0}$ such that $f_k: C_k(X)\rightarrow C_k(Y)$ is a homomorphism and $f_k \partial_{k+1}= \partial_{k+1} f_{k+1}$ for all $k\geq 0$, i.e. the following diagram commutes.

      \begin{equation*}\label{eq:chain_map}
          \xymatrix{
		\ldots \ar[r] & C_{k+1}(X)~\ar[r]^{\partial_{k+1}}\ar[d]^{f_{k+1}} & C_{k}(X)~\ar[d]^{f_k} \ar[r] &\ldots \\
		\ldots \ar[r] & C_{k+1}(Y)~\ar[r]^{\partial_{k+1}} & C_k(Y) \ar[r] & \ldots  \\
          }
      \end{equation*}

A continuous map $f: X\rightarrow Y$ induced by a map between the vertices of $X$ and $Y$ yields a simplicial chain map $f:C_*(X)\rightarrow C_*(Y)$, and a cellular map $f: X\rightarrow Y$ yields a cellular chain map $f: C^{\CW}_*(X)\rightarrow C^{\CW}_*(Y)$. Given two chain maps $f,g: C_*(X)\rightarrow C_*(Y)$, a $\textbf{chain homotopy}$ between them is a a sequence of maps $c=(c_k)_{k\geq 0}$, $c_k: C_k(X)\rightarrow C_{k+1}(Y)$ such that $c_{k-1} \partial_k+\partial_{k+1} c_k = f_k - g_k$, or  for short, $c \partial +\partial c=f-g$. This is equivalent to Diagram~\ref{eq:chain_map} commuting. We will drop the subscripts when the dimension is clear.
If a chain homotopy exists between $f$ and $g$, then $f$ and $g$ are said to be \textbf{chain homotopic}, written $f\simeq g$.

      \begin{equation*}\label{eq:chain_map}
          \xymatrix{
		\ldots \ar[r] & C_{k+1}(X)~\ar[r]^{\partial_{k+1}}\ar[d]_{f_{k+1}, g_{k+1}} & C_{k}(X)~ \ar[r]^-{\partial_k}\ar[ld]_{c_{k}}\ar[d]_{f_{k},g_{k}}  & C_{k-1}(X) \ar[r]\ar[d]^{f_{k-1},g_{k-1}}\ar[ld]_{c_{k-1}} &\ldots \\
		\ldots \ar[r] & C_{k+1}(Y)~\ar[r]^{\partial_{k+1}} & C_k(Y) \ar[r]^-{\partial_k} & C_{k-1}(Y) \ar[r] &\ldots  \\
          }
      \end{equation*}

Given a simplicial complex $X$ and its simplicial chain complex $C_*(X)$, the $k$-dimensional \textbf{simplicial homology group} as $H_k(X):=\ker\;\partial_k/\im\;\partial_{k+1}$.
The rank of this groups is the number of linearly independent $k$-dimensional holes in the space.
 Collectively, the homology groups of $X$ are denoted by $H_*(X)$.
  If two spaces have the same homotopy type, then their corresponding chain maps are chain homotopic, which then implies their homology groups are isomorphic in all dimensions.
 The \textbf{reduced homology} $\tilde{H}_*(X)$ is the homology computed from the chain complex $C_*(X)$ where one adjoins a copy of $\mathbb{Z}/2\mathbb{Z}$ to $C_{-1}(X)$.
 Functionally this means that $H_0(X)=\tilde{H}_0(X)\bigoplus \mathbb{Z}/2\mathbb{Z}$ so that the one-point space has trivial reduced homology groups over all dimensions.
 
Given two finite CW-complexes $X$ and $Y$, there is a natural isomorphism $C^{\CW}_*(X\times Y)\cong C^{\CW}_*(X)~\otimes~C^{\CW}_*(Y)$, where $\otimes$ is the tensor product.
In particular, $C^{\CW}_k(X\times Y) \cong  \bigoplus_{p+q=k} C^{\CW}_p(X)\otimes C^{\CW}_q(Y)$, where the basis of $C^{\CW}_k(X\times Y)$ is the collection of products $e^p\times f^q$, where $e^p$ is a $p$-cell in $X$ and $f^q$ is a $q$-cell in $Y$ for $p+q=k$.
Each basis elements $e^p\times f^q$ is identified under the isomorphism with $e^p\otimes f^q \in C^{\CW}_p(X)\otimes C^{\CW}_q(Y)$.  
The product chain complex $C_*^{\CW}(X\times Y)$ has the boundary map defined on any basis element $\sigma\otimes\tau \in C_k(X\times Y)$ by $\partial(\sigma\otimes \tau):= \partial( \sigma)\otimes \tau + \sigma\otimes \partial(\tau)\in C^{\CW}_{k-1}(X\times Y)$ and it extends linear to all chains.
 
The cellular homology of a CW-complex $H^{\CW}_*(X)$ is defined in the same manner as simplicial homology except instead with cellular chain complex groups and boundaries.
 In fact, for a finite simplicial complex $X$, $H^{\CW}_*(X)=H_*(X)$, which is a result of the cellular and simplicial chain complexes being canonically isomorphic --- the $k$-simplices are the $k$-cells when $X$ is viewed as a CW-complex.

\subsection{Filtrations and Persistence}

A \textbf{filtration} is a sequence of topological spaces $\mathcal{F} = (F^\alpha)_{\alpha\geq 0}$ such that $F^\alpha\subseteq F^\beta$ if and only if $\alpha\leq\beta$.
If a filtration consists of simplicial complexes, it is known as a simplicial filtration, and if each of the simplicial complexes is finite, it is known as a finite simplicial filtration.
Filtrations often arise as the sublevel sets of a real-valued function on a topological space or simplicial complex.

Persistent homology is the changes in the homology of a filtration as it ranges over the interval $[0,\infty)$.
To be precise, it is the computation of the ``birth" and ``death" scales of homological features under the homology maps induced by inclusion, $i_\mathcal{F}^{\alpha,\beta}:H_*(F^\alpha\hookrightarrow F^\beta)$, for all $\alpha,\beta$ such that $\alpha\leq \beta$.
The persistent homology data of a filtration is contained in its \textbf{persistence module}, denoted by $H_*(\mathcal{F})$, which consists of the spaces $H_*(F^\alpha)$ over all scales, and the aforementioned maps $i_\mathcal{F}^{\alpha,\beta}$ for $\alpha\leq\beta$.
The birth and deaths scales of $k$-dimensional homological features in a filtration $\mathcal{F}$ are represented in a filtration's $k$-dimensional \textbf{persistence diagram}, denoted by $\text{Dgm}_k(\mathcal{F})$.
This is a multiset with elements being points in the plane $(x,y)$, where $x$ and $y$ are the birth and death scales respectively of features, and $(x,x)$ for all $x\in\R$ with infinite multiplicity.
When discussing the persistence diagrams collectively for all dimensions we write $\text{Dgm}(\mathcal{F}):=(\text{Dgm}_k(\mathcal{F}))_{k\geq0}$.
The following theorem provides a condition under which we can say that two filtrations have identical persistence diagrams, often called the Persistence Equivalence Theorem --- see chapter 26 of~\cite{handbookcompgeo17}.
\begin{theorem}[Persistence Equivalence Theorem]\label{thm:persistent_lemma}
Consider two filtrations $\mathcal{F}=(F^\alpha)_{\alpha\geq 0}$, $\mathcal{G}=(G^\alpha)_{\alpha\geq 0}$ with point-wise finite dimensional persistence modules.
 If their persistence modules $H_*(\mathcal{F})$ and $H_*(\mathcal{G})$ are isomorphic then the filtrations have identical persistent homology and $\text{Dgm}(\mathcal{F})=\text{Dgm}(\mathcal{G})$.
\end{theorem}

Note that $H_*(\mathcal{F})$ is isomorphic to $H_*(\mathcal{G})$ if and only if there are natural isomorphisms $H_*(F^\alpha)\cong H_*(G^\alpha)$ for all $\alpha\geq 0$.

A persistence diagram is finite if it has finitely many off-diagonal points.
The standard metric on the space of persistence diagrams is the bottleneck distance $\dist_B$, which is efficiently computable for finite diagrams.
 For two finite diagrams $D$ and $D'$ it is defined as
\[
\dist_B(D,D'):=\min_{\phi\in\Phi} \max_{p\in D} \|p-\phi(p)\|_{\infty},
\]
where $\Phi$ is the set of all bijections $\phi:D\rightarrow D'$.
Two finite persistent diagrams are equivalent if and only if the bottleneck distance between them is $0$.
A major result in topological data analysis is that the bottleneck distance stable with respect to perturbations of the function generating the diagram, which is known as the Stability Theorem~\cite{cohen-steiner07stability}.

Given two filtrations $\mathcal{F}=(F^\alpha)_{\alpha\geq 0}$, $\mathcal{G}=(G^\alpha)_{\alpha\geq 0}$, their persistence modules $H_*(\mathcal{F})$ and $H_*(\mathcal{G})$ are \textbf{$\delta$-interleaved} if there exists collections of homomorphisms $f= (f^\alpha)_{\alpha\geq 0}$ and  $g: (g^\alpha)_{\alpha\geq 0}$, $f^\alpha: H_*(F^\alpha)\rightarrow H_*(G^{\alpha+\delta})$ and $g^\alpha : H_*(G^\alpha)\rightarrow H_*(F^{\alpha+\delta})$, such that
$g^{\alpha+\delta}f^\alpha=i^{\alpha,\alpha+2 \delta}_{\mathcal{F}}$ and
$f^{\alpha+\delta}g^\alpha=i^{\alpha,\alpha+2\delta}_\mathcal{G}$,
and these maps commute with all $i^{\alpha,\beta}_\mathcal{F}$ and $i^{\alpha,\beta}_\mathcal{G}$ for all $\alpha\leq \beta$.
This is known as an additive interleaving.
 Persistence module interleavings and their persistence diagrams' bottleneck distances are related by the Algebraic Stability Theorem (see~\cite[Thm~4.4]{chazal09proximity}), 

\begin{theorem}[Algebraic Stability Theorem]\label{thm:interleaving_bound}
Given two filtrations $\mathcal{F} = (F^\alpha)_{\alpha\geq 0}$ and $\mathcal{G} = (G^\alpha)_{\alpha\geq 0}$ such that for all $\alpha\geq 0$, $\dim H_*(F^\alpha),\dim H_*(G^\alpha)<\infty$, if $H_*(\mathcal{F})$ and $H_*(\mathcal{G})$ are $\delta$-interleaved then $\dist_B(\text{Dgm}_k(\mathcal{F}),\text{Dgm}_k(\mathcal{G}))\leq\delta$ for all $k$.
\end{theorem}

Let $\U := \{U_0,\ldots,U_n\}$ be a collection of simplicial filtrations, where $U_i := (U_i^\alpha)_{\alpha\ge 0}$ and for all $i$ and $\alpha$, $U_i^\alpha$ is a finite subcomplex of an ambient simplicial complex.
Let $\U^\alpha := \{U_0^\alpha,\ldots,U_n^\alpha\}$,  the collection of simplicial complexes at scale $\alpha$ from each filtration $U_i$.
For each non-empty $v\subseteq [n]$, let $U_v^\alpha:= \bigcap_{i \in v}U_i^\alpha$ yielding a simplicial filtration $U_v := (U_v^\alpha)_{\alpha\ge 0}$.
Note that in this notation $U_{\{i\}}^\alpha = U_i^\alpha$ and $U_{\{i\}} = U_i$.
For a collection of filtrations $\mathcal{U}$ there is an associated \textbf{nerve filtration} $\nerve \U := (\nerve (\U^\alpha))_{\alpha\ge 0}$.
For each scale the union of over the elements of $\U^\alpha$ is the simplicial complex defined as $W^\alpha := \bigcup_{i=0}^n U_i^\alpha$ and the the \textbf{union filtration} with respect to $\U$ is denoted by $\W := (W^\alpha)_{\alpha\ge 0}$.
For each $W^\alpha$, $\mathcal{U}^\alpha$ is a cover and thus we say that $\mathcal{U}$ is a cover filtration of $\mathcal{W}$, or cover for short.
We call $\U$ a \textbf{good cover filtration}, or good cover for short, of $\W$ if $\U^\alpha$ is a good cover of $W^\alpha$ for all $\alpha \ge 0$.

 The  previous definitions allow for the statement of the Persistent Nerve Lemma which the main theorem of this paper generalizes.
    The following lemma was originally formulated by Fr\'ed\'eric Chazal and Steve Oudot in \cite{chazal08towards} as a generalization of the Nerve Theorem to filtrations.

 \begin{lemma}\label{lem:persistent_nerve_lem}
 Let $X\subseteq X'$ be two finite simplicial complexes with good covers, $\mathcal{V}=\{V_\alpha\}_{\alpha\in A}$ and $\mathcal{V}'=\{V'_\alpha\}_{\alpha\in A}$ respectively, such that $V_\alpha\subseteq V'_\alpha$ for all $\alpha\in A$.
  There exists homotopy equivalences $|\nerve\mathcal{V}|\rightarrow X$ and $|\nerve\mathcal{V}'|\rightarrow X'$ that commute with the topological inclusions $X\hookrightarrow X'$ and $|\nerve\mathcal{V}|\hookrightarrow |\nerve\mathcal{V}'|$.
 \end{lemma}

Viewing each $\U^\alpha$ from a good cover filtration $\U$ as a good cover of $W^\alpha$,  Theorem~\ref{thm:persistent_lemma} can be applied to the construction of the homotopy equivalences in the proof of Lemma~\ref{lem:persistent_nerve_lem} to achieve the following fundamental persistent homology result.

 \begin{theorem}[Persistent Nerve Lemma]\label{thm:persistent_nerve_dgm}
Given a collection of finite simplicial filtrations $\U$  where $\U$ is a good cover filtration of $\W$, then $\text{Dgm}(\nerve \U)=\text{Dgm}(\W)$.
    \end{theorem}
\section{The Generalized Persistent Nerve Theorem}\label{sec:gen_nerve_thm}
  In the Persistent Nerve Lemma a primary assumption is that $\U^\alpha:= \{U_0^\alpha,\ldots, U_n^\alpha\}$ is a good cover of $W^\alpha:= \bigcup_{i=0}^n U_i^\alpha$ for all $\alpha$, i.e. $\U$ is a good cover filtration of $\W$.
  However, there are common situations where a simplicial cover may not be good as shown in the figures below.

  \begin{figure}[h]
  \label{fig:bump}
    \centering
    \includegraphics[width=0.5\textwidth]{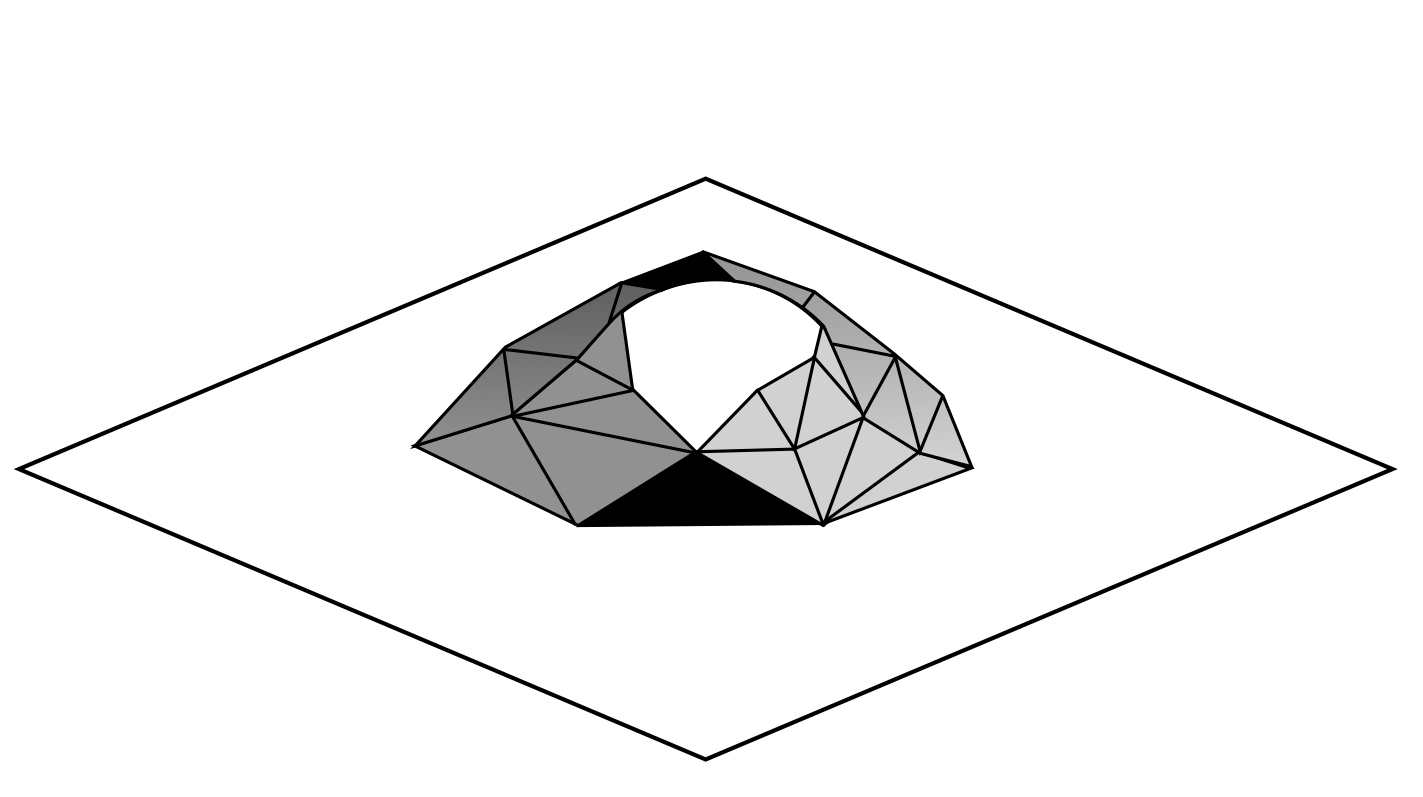}
    \caption{A highlighted portion of a triangulation of a surface with a bump where two simplicial complexes intersect in two connected components on the side (in black), thus they are not a good cover.
    Growing the two cover elements results in their intersection being contractible.}
  \end{figure}
  
  \begin{figure}[t]
  \label{fig:mesh}
    \centering
    \includegraphics[width=0.5\textwidth]{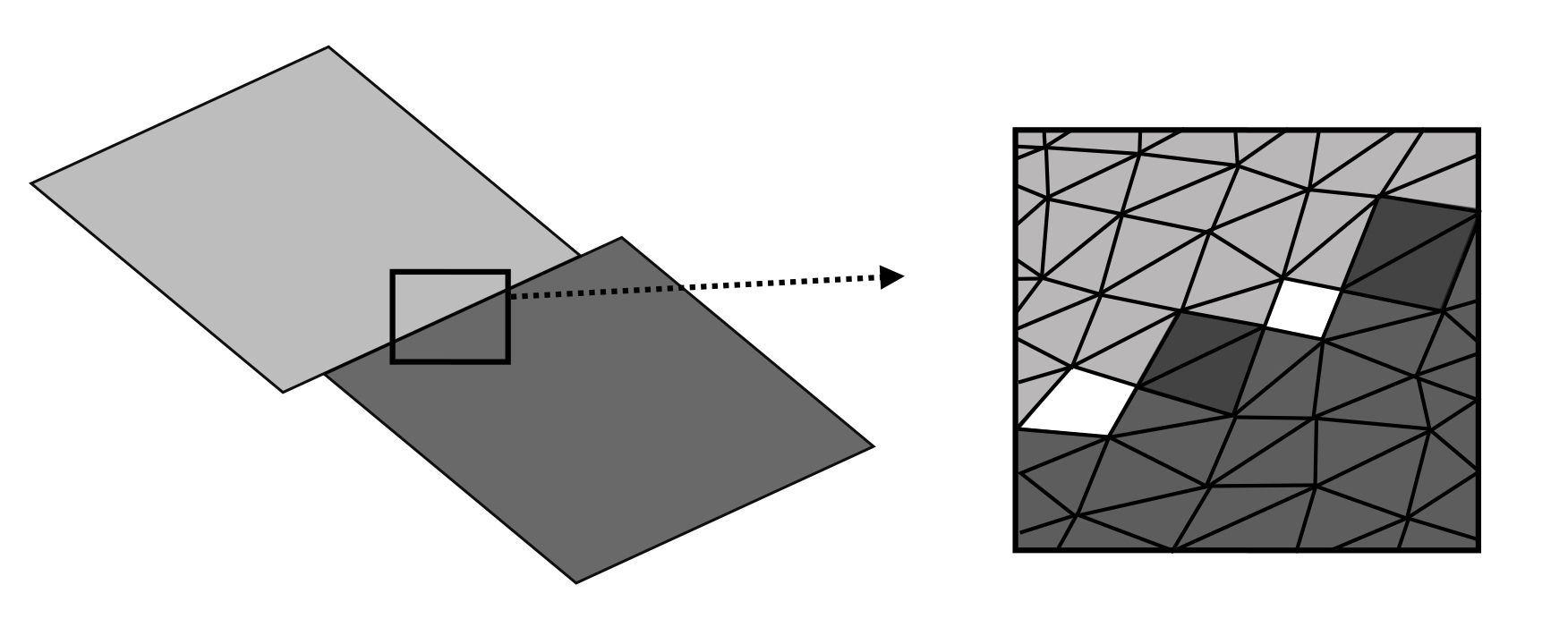}
    \caption{On the left is a space covered by two quadrilaterals.
    Zooming in on a triangulation of the cover elements, their intersection is not contractible by a small margin potentially due to an approximation error.}
  \end{figure}
  
  In both cases these good cover violations are relatively small and intuitively ought to be able to be made insignificant for purposes of, for example, recovering the homotopy type of a triangulated space.
  Persistent homology is an ideal theory for quantifying what is meant by a small violation of the good cover condition.
 The following is our natural generalization of a good cover filtration.
 
  \begin{definition}
  A cover filtration $\mathcal{U}$ is an \textbf{$\e$-good cover} of a filtration $\mathcal{W}=(\bigcup_{i=0}^nU_i^\alpha)_{\alpha\geq 0}$ if for all non-empty $v\subseteq[n]$, and all $\alpha\geq 0$, $\tilde{H}_*(U_v^\alpha\hookrightarrow U_v^{\alpha+\e})=0$.
 \end{definition}

  \begin{figure}[t]
    \centering
    \includegraphics[width=0.6\textwidth]{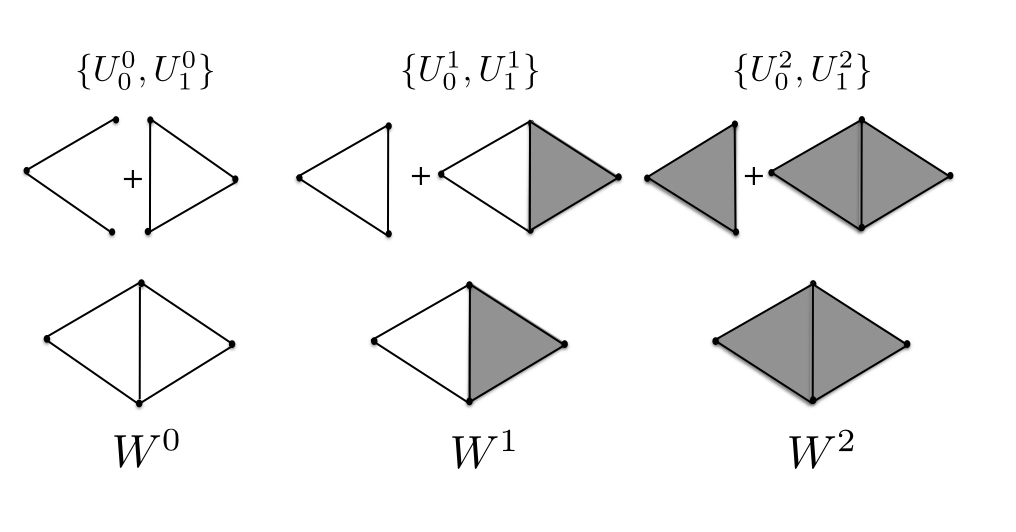}
    \caption{The simplicial cover filtration $\U$ is not a good cover at scales $0$ and $1$, but is $1$-good.
    The images of the inclusions from each intersection of covers has trivial homology at the next scale.}
    \label{fig:eps_good_cover}
  \end{figure}

Though the definition of an $\e$-good covers is stated in terms of homology of the inclusions of cover intersections, it is in fact weaker than the assumption that $\U^\alpha_v\hookrightarrow U^{\alpha+\e}_v$ is null-homotopic, which would align better with the traditional notion of a good cover.
We choose to still use the term ``good" despite this choice.
The main goal of this paper is providing a theorem concerning persistent homology so nothing would be gained by working at the homotopy level.

  There is the following relation between the definitions of an $\e$-good cover and a good cover.
  If $\U$ is a good finite simplicial cover then for each each non-empty intersection of cover elements $U_v^\alpha$ of $\U^\alpha$,  $U_v^\alpha$ is homotopy equivalent to a point.
  This implies that $\im~ \tilde{H}_*(U_v^\alpha\hookrightarrow U_v^{\alpha+\e})\subseteq \tilde{H}_*(U_v^{\alpha+\e})= 0$, so $\U$ is a $0$-good cover.
However, the converse does not hold --- the $2$-skeleton of the Poincar\'e $3$-sphere has trivial reduced homology groups but is not contractible.

  Recall the definition of the nerve filtration $\nerve\U:= (\nerve\U^\alpha)_{\alpha\geq 0}$.
  The following theorem called the Generalized Persistent Nerve Theorem provides a tight bound  of $(K+1)\e$ on the bottleneck distance between the $K$-dimensional persistence diagrams of the nerve filtration and a simplicial cover filtration, given that $\U$ is an $\e$-good cover filtration of $\W$.

  \begin{theorem}[Generalized Persistent Nerve Theorem]\label{thm:gpnl}
      Given a finite collection of finite simplicial filtrations $\U = \{U_0,\ldots,U_n\}$, where  $U_i:=(U_i^\alpha)_{\alpha\geq 0}$ and all $U_i^\alpha$ are subcomplexes of a sufficiently large simplicial complex,  if $\U$ is an $\e$-good cover filtration of $\W := (\bigcup_{i=0}^n U_i^\alpha)_{\alpha\geq 0}$, then
      \[
        \dist_B(\text{Dgm}_K(\W), \text{Dgm}_K(\nerve \U)) \le (K+1)\e.
      \]
  \end{theorem}

As Theorem~\ref{thm:gpnl} is true for all dimensions it implies that $\dist_B(\text{Dgm}(\W),\text{Dgm}(\nerve\U))\leq (D+1)\e$ for $D :=\dim\nerve \U$, as well as implying Theorem~\ref{thm:persistent_nerve_dgm} (the Persistent Nerve Lemma) for the case of finite simplicial filtrations.
See Appendix~\ref{sec:tight_bound} for a construction that realizes the bottleneck distance bound over all dimensions.

The Generalized Persistent Nerve Theorem can be seen as an extension of the Persistent Nerve Lemma analogous to how the Algebraic Stability Theorem for persistence modules extends the Persistence Equivalence Theorem to interleaved modules, by viewing $\e$-good cover filtrations as perturbations or approximations of good cover filtrations, the ideal object.
 This relationship is summarized in the following table.
 \\
\begin{center}
\begin{tabular}{ | c || c | c | }
\hline
       & \emph{Equivalence} & \emph{Approximation} \\
\hline\hline
 \emph{Persistence Modules} & Persistence Equivalence Theorem & Stability Theorem \\
\hline
\emph{Nerves} & Persistent Nerve Lemma & \textbf{Gen. Persistent Nerve Theorem} \\
\hline
\end{tabular}
\end{center}

\section{Proof Construction}
\label{sec:proof_construction}
 Fix a cover filtration $\U := \{U_0,\ldots,U_n\}$ consisting of finite simplicial filtrations, where $U_i:= (\bigcup_{i\in[n]} U_i^\alpha)_{\alpha\ge 0}$ and all the $U_i^\alpha$ are subcomplexes of some sufficiently large simplicial complex.
 Assume that $\U$ is an $\e$-good cover filtration of the simplicial filtration $\W: = (W^\alpha)_{\alpha\geq 0} := (\bigcup_{i=0}^n U_i^\alpha)_{\alpha\geq 0}$.
 Fix a dimension $K$ for the remainder of the proof which will the be maximal dimension considered when discussing chain complexes, i.e. $C_*(X):=(C_k(X))_{k\leq K}$ for any space $X$ and likewise for homology groups.
 
The procedure to prove Theorem~\ref{thm:gpnl} is as follows. 
First we construct a diagram of chain complexes and chain maps that yield a $(K+1)\e$-interleaving between the filtered chain complexes $C_k(\W)$ and $C_k(\nerve \U)$ for all $k\leq K$.
This chain complex interleaving is analogous to that defined previously between persistence diagrams, except one we only require the appropriate compositions be chain homotopic to the identity chain maps rather than equivalent.
Applying homology to said chain complex diagram, the chain maps that are chain homotopic to the identity become homologically equivalent to identity maps so there is a true $(K+1)\e$-interleaving between the persistence modules $H_k(\W)$ and $H_k(\nerve\U)$ for each $k\leq K$.
The theorem is proved then by applying the Algebraic Stability Theorem (Theorem~\ref{thm:interleaving_bound}).

\subsection{The Nerve Diagram and Blowup Complex}\label{sec:nerve_diagram}
For each $\alpha \ge 0$, define $N^\alpha$ as the directed graph with vertex set $\{ \textrm{non-empty } v \subseteq [n] \mid U_v^\alpha := \bigcap_{i\in v} U_i^\alpha\neq\emptyset \}$, and edges of the form $(v',v)$ for any non-empty $v, v'\subseteq [n]$ such  $v\subset v'$.
Note the vertices of this graph are in correspondence with the simplices of $\nerve\U^\alpha$ and thus form the $0$-skeleton of the barycentric subdivision of $\nerve \U^\alpha$, while the entire 
(undirected) graph is its $1$-skeleton.
The edges correspond to inclusions between intersections of cover elements of the form $U_{v'}^\alpha\hookrightarrow U_v^\alpha$.

$N^\alpha$ can be given the structure of an abstract simplicial complex, where the $k$-simplices are sequences of vertices $v_0,\ldots,v_k$ such that for $i$ and $j$, $0\leq i<j\leq k$, there exists an edge $(v_i, v_j)$.
Some call this complex built from an acyclic graph or poset a flag complex.
From now on we will use the notation $N^\alpha$ to refer to the simplicial representation.
A fact that will be important later on is that its geometric realization $|N^\alpha|$ is homeomorphic as a topological space to $|\nerve\U^\alpha|$.
\begin{figure}[h]
\centering
    \includegraphics[width=.9\textwidth]{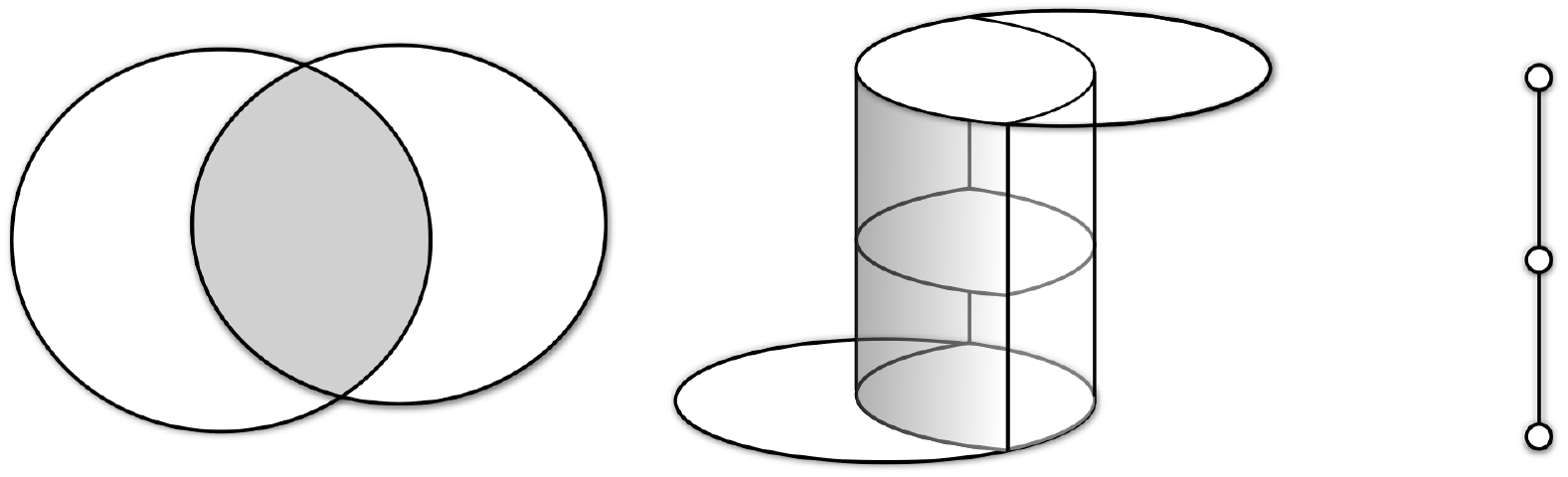} 
  \caption{(L) The barycentric decomposition of the blow-up complex of a two-element cover alongside the decomposition of their nerve.}
  \end{figure}

From $N^\alpha$ and $W^\alpha$ we define the finite CW-complex $B^\alpha$ that glues together all the realizations of the simplices of $N^\alpha$ paired with their corresponding cover elements' intersection in $W^\alpha$.
     \[
       B^\alpha := \bigcup_{\substack{N^\alpha\ni \sigma = v_0\to\cdots\to v_k \\ k\geq 0}} U_{v_0}^\alpha\times |\sigma|.
    \]
    
Note that this is the barycentric decomposition of the so-called (Mayer--Vietoris) blow-up complex central to the proof of the Nerve Theorem, called the realization of a diagram of spaces in Hatcher~\cite{Hatcher01}, and other more recent persistent homology research, e.g. Zomorodian and Carlsson's work on localized homology~\cite{zomorodian08localized}.
Other readers familiar with combinatorial topology and discrete geometry may recognize this as the homotopy colimit of the categorical diagram between $N^\alpha$ and $\text{Top}$ constructed from the cover elements' interesections' correspondence with the vertices and the inclusions with the edges of $N^\alpha$.
One may refer to Kozlov~\cite{koslov08}, p.262, for the relevant definition or Welker et al.'s treatise on homotopy colimits and their applications ~\cite{welker99homotopy}.
This is expanded upon in Appendix~\ref{sec:hom_colimit}.

The associated filtration is denoted $\B:=(B^\alpha)_{\alpha\ge 0}$.
By definition we have that $B^\alpha \subseteq W^\alpha \times | N^\alpha |$ for all $\alpha$.
 As the filtration $\B$  organizes and combines the nerves and the covered simplicial complexes, it is easier to define maps from it rather than from $\W$.

\begin{figure}
\centering
    \includegraphics[width=.9\textwidth]{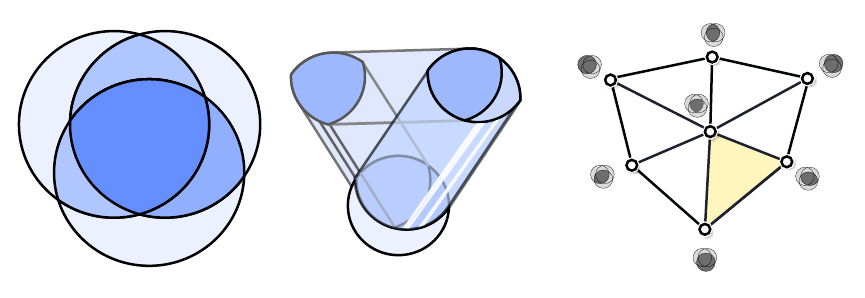} 
  \caption{(L) A portion of the blowup complex for a three-element cover.
  Highlighted is a $2$-simplex in the realization of the barycentric decomposition of the nerve and its associated cell in the blowup.}
  \end{figure}
  
  Using the filtration $\B$ we will now reduce the proof to constructing a $(K+1)\e$-interleaving between $H_K(\B)$ and $H_K(\nerve\U)$ as follows.
    There are natural projection maps $b^\alpha:=\pi_{W^\alpha} i$ for each $\alpha\geq 0$, where $i:B^\alpha \hookrightarrow W^\alpha \times | N^\alpha |$ and $\pi_{W^\alpha}: W^\alpha \times | N^\alpha |\rightarrow W^\alpha$ when 
    It is well-known fact that that $b^\alpha$ is a homotopy equivalence (see 4G.2 in~\cite{Hatcher01}) for covers of paracompact spaces, so $H^{\CW}_*(B^\alpha)\cong H^{\CW}_*(W^\alpha)\cong H_*(W^\alpha)$ as finite simplicial complexes are paracompact.
    
        Moreover, two projections $b^\alpha$ and $b^\beta$ commute with the inclusions $i_\mathcal{B}^{\alpha,\beta}: B^\alpha\hookrightarrow B^\beta$ and $i_\mathcal{W}^{\alpha,\beta}: W^\alpha\hookrightarrow W^\beta$, yielding the following commutative homological diagram for all scales $\alpha,\beta$ such that $0\leq\alpha\leq \beta$.
  \begin{center}
          \begin{equation}\label{eq:b_map}
          \xymatrix{
		 H_{*}(W^\alpha)~\ar[r]^{i_\mathcal{W}^{\alpha,\beta}}  & H_*(W^\beta) \\
		H^{\CW}_*(B^\alpha)~\ar[r]^{i_\mathcal{B}^{\alpha,\beta}} \ar[u]_\cong^{b^\alpha}& H^{\CW}_*(B^\beta)~\ar[u]_\cong^{b^\beta}\\
          }
      \end{equation}
      \end{center}
    
 From Diagram~\ref{eq:b_map} it follows that $\text{Dgm}(\B) = \text{Dgm}(\W)$ by Theorem~\ref{thm:persistent_lemma}.

    Define the filtration $\mathcal{N}:=(|N^\alpha|)_{\alpha\ge 0}$.
    There are also natural projection maps $p^\alpha:=\pi_{|N^\alpha|} i$ for each $\alpha\geq 0$, where $i:B^\alpha \hookrightarrow W^\alpha \times | N^\alpha |$ and $\pi_{N^\alpha}: W^\alpha \times | N^\alpha |\rightarrow |N^\alpha|$.
  When $\U^\alpha$ is a good cover of $W^\alpha$, and by extension $\U$ is a good cover filtration of $\W$, the projection maps $p^\alpha$ are homotopy equivalences that commute with the filtration inclusions.
  This is a central component to the proof the Nerve Theorem and Lemma~\ref{lem:persistent_nerve_lem}.
      
    However, under our assumptions, we do not have the homotopy equivalences resulting from the good cover condition, so instead we must find an $(K+1)\e$ interleaving between $H_K(\mathcal{N})$   and $H^{\CW}_K(\mathcal{B})$ to prove our theorem.
    As $|N^\alpha|$ is homeomorphic to $|\nerve\U^\alpha|$, this is sufficient.
   In the next section we will construct the chain maps mapping a basis chain in $C_K^{\CW}(B^\alpha)$ to a chain in $C_K^{\CW}(W^{\alpha+(K+1)\e})$ which will result in an interleaving by symmetrizing.
    
\subsection{The Chain Maps}

     For the remainder of the proof we will not use the geometric realization vertical bars when discussing basis cellular chains corresponding to tensor products of geometric simplices viewed as cells to avoid cumbersome presentation.
     For an abstract $k$-simplex $\sigma\in N^\alpha$ where $\sigma= v_0\to\ldots\to v_k$, the following two shorthands will be used, $\sigma_i:=|\sigma _{[v_0,\ldots,v_i]}|$ and $\overline{\sigma_i}:=|\sigma _{[v_i, \ldots, v_k]}|$.
     These are the geometric realizations of the restriction of $\sigma$ to the first $i$-vertices and the $i$-th through $k$-th vertices respectively.
 For notational simplicity, given some $\sigma$ and some vertex $v\in \sigma$, $\sigma\setminus v:=|\sigma _{[v_0,\ldots, \hat{v},\ldots, v_k]}|$, where $\hat{v}$ denotes the removal of vertex $v$.

 For each non-empty $v\subseteq [n]$, pick a vertex $x_v\in U_v^{\alpha'}\subset W^{\alpha'}$ where $\alpha':=\min\{\alpha\ge 0\mid U_v^\alpha\neq\emptyset\}$, and note that $x_v\in U_v^\beta\subset W^\beta$ for all $\beta\geq \alpha'$.
 Consider the map $x_v:U_v^\alpha \to U_v^{\alpha+\e}$ which is the  extension of the linear map sending each vertex of $U_v^\alpha$ to to $x_v\in U_v^{\alpha+\e}$.
 This results in the chain map $x_v: C^{\CW}_*(U_v^\alpha) \to C^{\CW}_*(U_v^{\alpha+\e})$ defined as follows.
            \[
 x_v(\sigma) :=\begin{cases}
  x_v & \text{ if } \dim\sigma = 0,\\
  0 & \text{ otherwise.}
  \end{cases}
  \]
 
The following lemma will be used to show the existence of a chain homotopy between the inclusion chain maps and the vertex chain maps.

 \begin{lemma}\label{lem:exists_hom}
 Fix $\alpha\geq 0$.
 Given non-empty $v\subseteq [n]$, where $v\in\nerve\U^\alpha$, and $x_v\in U_v^\alpha$ is $v$'s corresponding vertex, consider the inclusion chain map $i_v^{\alpha,\alpha+\e}: C^{\CW}_*(U_v^\alpha) \rightarrow C^{\CW}_*(U_v^{\alpha+\e})$ and $x_v: C^{\CW}_*(U_v^\alpha)\rightarrow C^{\CW}_*(U_v^{\alpha+\e})$.
 There exists a chain homotopy from $i_v^{\alpha,\alpha+\e}$ to $x_v$.
 \end{lemma}
 \begin{proof}
 We construct the chain homotopy by induction on dimension to prove that for all $n$ that there exists $c_n,c_{n-1}$ such that $c_{n-1}\partial_n +  \partial_{n+1}c_n= i_v^{\alpha,\alpha+\e}-x_v$.
 In the base case we can consider $c_{-1}=c_{-2}=0$.
 Now for some $k\geq 0$, assume there exists $c_{k-1},c_{k-2}$ such that $c_{k-2}\partial_{k-1} +  \partial_{k}c_{k-1}= i_v^{\alpha,\alpha+\e}-x_v$.
 
Let $z$ be the $k$-chain such that $z=-(c_{k-1}\partial_k)(\sigma)+i_v^{\alpha,\alpha+\e}(\sigma)-x_v(\sigma)$.
Observe that for $k\geq 0$, for any $k$-simplex $\sigma\in C^{\CW}_k(U_v^\alpha)$,
 \begin{align*}
 \partial_k(z)&=-(\partial_kc_{k-1}\partial_k)(\sigma)+(\partial_k i_v^{\alpha,\alpha+\e})(\sigma)-(\partial_k x_v)(\sigma)\\
 &=(c_{k-2}\partial_{k-1}\partial_{k})(\sigma)-(i_v^{\alpha,\alpha+\e} \partial_k)(\sigma)+(x_v \partial_k)(\sigma)+(\partial_k i_v^{\alpha,\alpha+\e})(\sigma)-(\partial_k x_v)(\sigma)\\
 &=0
 \end{align*}
 
 The second line follows by the inductive hypothesis as $\partial_k(\sigma)$ is a $(k-1)$-chain. The third line follows as $\partial_{k-1}\partial_k=0$ and $i_v^{\alpha,\alpha+\e}$ and $x_v$ are chain maps so they commute with the boundary operators.

This proves that $z$ is a cycle so $i_v^{\alpha,\alpha+\e}(z)=z$ is a cycle, and $\tilde{H_*}(i_v^{\alpha,\alpha+\e}(z))=0$ as $\U$ is an $\e$-good cover.
There then must exist a boundary $b$ such that $\partial_{k+1}(b)=z$.
 Define $c_k(\sigma):=b$.
 By the above calculations this choice of $c_k$ satisfies the inductive hypothesis so we are done.
 \end{proof}

By Lemma~\ref{lem:exists_hom}, for a given $\alpha\geq 0$ and $v\in \nerve \U^\alpha$, there exists a chain homotopy, which we denote as $c^\alpha_v$, between the identity chain map $i_v^{\alpha,\alpha+\e}: C^{\CW}_*(U_v^\alpha)  \rightarrow C^{\CW}_*(U_v^{\alpha+\e})$  and the constant chain map $x^\alpha_v: C^{\CW}_*(U_v^\alpha)\rightarrow C^{\CW}_*(U_v^{\alpha+\e})$.
 By definition we have the equality $\partial c^\alpha_v + c^\alpha_v\partial = i_v^{\alpha,\alpha+\e} + x_v^\alpha$.

Denote $t:=(K+1)\e$ for short for the remainder of the paper when referencing the 
 For $k\leq K$, define the map $c^\alpha: C^{\CW}_k(B^\alpha)\rightarrow C^{\CW}_{k+1}(W^{\alpha+t})$ for a cellular basis element $\tau\otimes\sigma \in C_k^{\CW}(B^\alpha)$, where $\tau$ is a $p$-cell and $\sigma$ is a $q$-cell corresponding to the geometric realization of the abstract $q$-simplex $\sigma=v_0\to\ldots\to v_k$ in $N^\alpha$ , as
 \[
 c^\alpha(\tau\otimes\sigma):= c^\alpha_\sigma(\tau) := (c^{\alpha+q\e}_{v_q}\ldots c^\alpha_{v_0})(\tau).
 \] 

 Note that this is well-defined despite $c_v^\alpha$'s domain each being $C^{\CW}_*(U_v^\alpha)$ as for any basis cellular chain $\tau\otimes\sigma$, $\tau$ is a simplex of $U_v^\alpha$ for some $v$ by the definition of $B^\alpha$.
 
 Recall the two projections from the barycentric decomposition of the blow-up complex $B^\alpha$, $b^\alpha$ and $p^\alpha$.
 For a cellular chain $k$-chain $\tau\otimes\sigma\in C_k^{\CW}(B^\alpha)$, the projection-induced chain maps $b^\alpha: C_*^{\CW}(B^\alpha)\rightarrow C_*^{\CW}(W^\alpha)$ and $p^\alpha: C_*^{\CW}(B^\alpha)\rightarrow C_*^{\CW}(|N^\alpha|)$ are defined as follows,
 
  \[
  b^\alpha(\tau\otimes\sigma) := \begin{cases}
  \tau & \text { if } \dim\sigma = 0 \\
  0 & \text{ otherwise,}
  \end{cases}
  \]
  
  and
  
 \[
 p^\alpha(\tau\otimes\sigma) :=\begin{cases}
  \sigma & \text{ if } \dim\tau = 0,\\
  0 & \text{ otherwise.}
  \end{cases}
  \] 
  
Define the chain map $q^\alpha: C^{\CW}_*(|N^\alpha|)\to C^{\CW}_*(W^{\alpha+t})$ for a basis $k$-simplex $\sigma\in C^{\CW}_k(|N^\alpha|)$ to be the following

 \[
 q^\alpha(\sigma):=\begin{cases}
  c^\alpha(x_{v_0}\otimes (\sigma\setminus v_0)) & \text{ if }\dim\sigma \geq1\\
  x_{v_0} & \text{ if }\dim\sigma = 0.
  \end{cases}
  \] 
 
The chain maps $b^\alpha$ and $q^\alpha$ are induced by topological maps so they are chain maps by construction, while $q^\alpha$ it is not immediately apparent it is a chain map.
  \begin{lemma}\label{q_chainmap}
The map $q^\alpha: C^{\CW}_*(|N^\alpha|)\rightarrow C^{\CW}_*( W^{\alpha+t})$ as defined above, where $t=(K+1)\e$  is a chain map for all dimensions less than or equal to $K$.
  \end{lemma}
  
  \begin{proof}
  Denote $q^\alpha_k: C^{\CW}_k(|N^\alpha|)\rightarrow C^{\CW}_k( W^{\alpha+t})$ for this proof to make it clear what dimension is being worked in.
  We will prove that $q^\alpha$ is a chain map by induction on the basis  of $C^{\CW}_k(|N^\alpha|)$ for arbitrary $\alpha \geq 0$ and $k\leq K$.
  These are the simplices of $|N^\alpha|$ resulting from abstract simplices in $N^\alpha$ of the form $\sigma = v_0\to\ldots\to v_k$.
  
  In the base case, where $\dim\sigma=0$, then $\sigma= |v|$ for some vertex $v\in N^\alpha$, so we have that $\partial q^\alpha _0(v) = \partial(x_v) = 0 = q^\alpha_{-1}\partial(v)$.
  
  Now assume that for some $k\geq 1$, the following holds for any given basis $k$-chain $\sigma\in C^{\CW}_k(|N^\alpha|)$, $q_{k-1}^\alpha \partial_k(\sigma) = \partial_k q_k^\alpha(\sigma)$.
  Now consider a basis $(k+1)$-chain $\sigma' \in C^{\CW}_{k+1}(|N^\alpha|)$.
  We have the following equalities, defining $\sigma'' := \sigma'\setminus v_{k+1}$.
  
  \begin{align*}
  \partial q_{k+1}^\alpha(\sigma')&=\partial c^\alpha(x_{v_0}\otimes (\sigma'\setminus v_0))\\
  &= \partial c_{v_{k+1}}^{\alpha+k\e} c^\alpha(x_{v_0}\otimes (\sigma'\setminus (v_0\cup v_{k+1}))) \\
  &= (c_{v_{k+1}}^{\alpha+k\e}\partial + i_{v_{k+1}}^{\alpha+k\e,\alpha+ (k+1)\e} + x^\alpha_{v_{k+1}})(c^\alpha(x_{v_0}\otimes  (\sigma'\setminus (v_0\cup v_{k+1})))) \\
  &= c_{v_{k+1}}^{\alpha+k\e}\partial c^\alpha (x_{v_0}\otimes (\sigma''\setminus v_0)) + c^\alpha (x_{v_0}\otimes (\sigma''\setminus v_0)) + 0\\
  &= c_{v_{k+1}}^{\alpha+k\e}\partial q_k^\alpha(\sigma'') + q_k^\alpha(\sigma'')\\
  &= c_{v_{k+1}}^{\alpha+k\e} q_{k-1}^\alpha \partial(\sigma'') +q_k^\alpha(\sigma'') \\
  &= \sum^k_{j=0} c_{v_{k+1}}^{\alpha+k\e} q_{k-1}^\alpha(\sigma''\setminus v_j) + q_k^\alpha(\sigma'')  \\
  &= \sum^k_{j=0} q_k^\alpha (\sigma'\setminus v_j) +q_k^\alpha(\sigma'\setminus v_{k+1}) = \sum_{i=0}^{k+1} q_k^\alpha(\sigma'\setminus v_i) = q_k^\alpha \partial(\sigma').
  \end{align*}
  
We thus have by definition $q^\alpha$ is a chain map.
 \end{proof}

 Now we will define the last chain map needed $a^\alpha: C^{\CW}_*(B^\alpha)\to C^{\CW}_*(W^{\alpha+t})$, the composition of $q^\alpha$ and $p^\alpha$.
 As $p^\alpha$ and $q^\alpha$ is a chain map, $a^\alpha$ is a chain map as it is the composition of chain maps.
 For any basis cellular chain $\tau\otimes\sigma\in C_k^{\CW}(B^\alpha)$,

  \[
  a^\alpha(\tau\otimes\sigma)=(q^\alpha p^\alpha)(\tau\otimes\sigma)=\begin{cases}
  c^\alpha(x_{v_0}\otimes(\sigma\setminus v_0)) & \text{ if } \dim\sigma\geq 1 \text{ and } \dim\tau =0\\
  x_{v_0} & \text{ if } \dim\sigma = 0 \text{ and } \dim\tau = 0 \\
  0 & \text{ otherwise,}
  \end{cases}
  \]

Figure~\ref{eq:complete_dgm} is a diagram of the $k$-dimensional cellular chain groups in question and the maps between them we have just defined.
  
         \begin{equation}\label{eq:complete_dgm}
          \xymatrix{
		 C^{\CW}_{k}(W^\alpha)~\ar[r]^{i_\mathcal{W}^{\alpha,\alpha+t}} & C^{\CW}_k(W^{\alpha+t}) \\
		 C^{\CW}_{k}(B^\alpha)~\ar[r]^{i_\mathcal{B}^{\alpha,\alpha+t}}\ar[u]^{b^\alpha} \ar[d]_{p^\alpha} & C^{\CW}_k(B^{\alpha+t})~\ar[d]^{p^{\alpha+t}} \ar[u]_{b^{\alpha+t}}\\
		 C^{\CW}_{k}(|N^\alpha|)~\ar[ruu]_(.42){q^\alpha}\ar[r]^{i_\mathcal{N}^{\alpha,\alpha+t}} & C^{\CW}_k(|N^{\alpha+t}|) \\
          }
      \end{equation}
      
In the following lemma, we prove that $a^\alpha$ and $i_\W^{\alpha,\alpha+t}b^\alpha$ are chain homotopic with chain homotopy $c^\alpha$ between them, which allows us to extend the local chain homotopies as guaranteed by Lemma~\ref{lem:exists_hom} to be defined on the entire blow-up complex.
This is a critical step in the proof as it implies one can extract global information from a seemingly local assumption, the $\e$-goodness of the cover.
      
  \begin{lemma}\label{lem:c_homotopy}
  $c^\alpha$ is a chain homotopy between the chain maps $a^\alpha$ and $i_{\W}^{\alpha,\alpha+t}b^\alpha: C^{\CW}_k(B^\alpha)\rightarrow C^{\CW}_k(W^{\alpha+t})$ for all $k \leq K$, i.e. $\partial c^\alpha = c^\alpha\partial +i_{\W}^{\alpha,\alpha+t}b^\alpha + a^\alpha$.
  \end{lemma}
  \begin{proof}
 
 First we check its true for the initial cases and then prove it for all dimensions by induction.
 
 Consider a basis cellular chain $\tau\otimes\sigma\in C^{\CW}_0(B^\alpha)$ where $\dim\tau=0$ and $\sigma= v_0$ some vertex $v_0\in N^\alpha$.
 We have the following equalities.

 \begin{align*}
 \partial c^\alpha(\tau\otimes v_0)=\partial c^\alpha_{v_0}(\tau) &= c^\alpha_{v_0}(\partial\tau) + \tau + x_{v_0}^\alpha(v_0) \\
&= 0 + \tau+ x_{v_0}\\
&= c^\alpha\partial (x\otimes v)+ b^\alpha(\tau\otimes v_0) + a^\alpha(\tau \otimes v_0). 
\end{align*}

The next case is if $\dim\tau>0$ and $\sigma = |v_0|$ for some vertex $v_0\in N^\alpha$.

\begin{align*}
 \partial c^\alpha(\tau\otimes v_0)=\partial c^\alpha_{v_0}(\tau)&=c^\alpha_{v_0}(\partial\tau)+\tau+x_{v_0}^\alpha(\tau) \\
 &=c^\alpha(\partial\tau\otimes v) + \tau + 0 \\
 &=c^\alpha\partial(\tau\otimes v) + b^\alpha(\tau\otimes v_0) + a^\alpha(\tau\otimes v_0)
 \end{align*}

Now assume that $\dim\tau =0$ and $\dim\sigma = 1$, so that $\sigma = v_0\rightarrow v_1$.

\begin{align*}
\partial c^\alpha(\tau\otimes (v_0\rightarrow v_1))=\partial(c^{\alpha+\e}_{v_1}(c^\alpha_{v_0}(\tau))) &= (c^{\alpha+\e}_{v_1}\partial + i_{v_1}^{\alpha+\e,\alpha+2\e} + x_{v_1}^{\alpha+\e})(c^\alpha_{v_0}(\tau))\\
&= c_{v_1}^{\alpha+\e}\partial c_{v_0}^\alpha(\tau) + c_{v_0}^\alpha(\tau) + 0 \\
&= c^\alpha_{v_1}c^\alpha_{v_0}(\partial \tau) + \tau + x_{v_0}) + c^\alpha_{v_0}(\tau)\\
&= c^\alpha(\partial \tau\otimes \sigma) + c^\alpha_{v_1}(\tau) + c^\alpha_{v_1}(x_{v_0}) + c^\alpha_{v_0}(\tau)\\
&= c^\alpha(\partial \tau\otimes \sigma) + c^\alpha(\tau\otimes v_1) + c^\alpha(\tau \otimes v_0) + c^\alpha(x_{v_0}\otimes v_1) \\
&= c^\alpha(\partial \tau\otimes \sigma) + c^\alpha(x\otimes \partial\sigma) + a^\alpha(\tau\otimes \sigma)\\
&= c^\alpha\partial (\tau \otimes \sigma) + a^\alpha(\tau\otimes \sigma) + b^\alpha(\tau\otimes \sigma)
\end{align*}

Now assume that for any basis chain $\tau\otimes\sigma$ such that $\dim\tau=0$ and $\dim\sigma\leq k-1$ , $\partial c^\alpha(\tau\otimes\sigma)=\partial c^\alpha + i_{\W}^{\alpha,\alpha+t}b^\alpha + a^\alpha$. 
For a $k$-simplex $\sigma\in |N^\alpha|$ we have the following equalities.

\begin{align*}
\partial c^\alpha(\tau\otimes \sigma) &= \partial c^\alpha_{v_k} (c^\alpha(\tau\otimes \sigma_{k-1}))\\
&= c^\alpha_{v_k}\partial (c^\alpha(\tau\otimes\sigma_{k-1})) + c^\alpha(\tau\otimes\sigma_{k-1}) + x_{v_k}^\alpha(\tau) \\
&= c^\alpha_{v_k}(c^\alpha\partial (\tau\otimes\sigma_{k-1}) + a^\alpha(\tau\otimes\sigma_{k-1}) + b^\alpha(\tau\otimes\sigma_{k-1}))+ c^\alpha(\tau\otimes\sigma_{k-1}) \\
&= c^\alpha_{v_k}(c^\alpha(\partial \tau\otimes\sigma_{k-1})+c^\alpha(\tau\otimes\partial(\sigma_{k-1})) +c^\alpha_{v_k}(a^\alpha(\tau\otimes\sigma_{k-1})+ c^\alpha_{v_k}(b^\alpha(\tau\otimes\sigma_{k-1}))))\\
& \text{  \indent\indent    } + c^\alpha(x\otimes\sigma_{k-1})  \\
&= c^\alpha(\partial \tau\otimes \sigma ) + c^\alpha(\tau\otimes\partial\sigma)  + 2c^\alpha(\tau \otimes\sigma_{k-1}) + c^\alpha_{v_k}(c^\alpha(x_{v_0}\otimes(\sigma_{k-1}\setminus v_0)) + 0 \\
&=  c^\alpha\partial(\tau\otimes\sigma) + 0 + c^\alpha(x_{v_0}\otimes(\sigma\setminus v_0)) \\
&= c^\alpha\partial(\tau\otimes\sigma) + b^\alpha(\tau \otimes\sigma) + a^\alpha(\tau \otimes\sigma).
\end{align*}

To complete the proof by induction, consider $\tau$ and $\sigma$ such that $\dim\tau=s\geq 1$ and $\dim\sigma = k \geq 2$, and assume the hypothesis holds true for all $\sigma'\in $ such that $\dim\sigma'=k-1$.
In particular this means it is true for $\sigma_{k-1}$ so the following holds. 

\begin{align*}
\partial c^\alpha(\tau\otimes \sigma) &= \partial c^\alpha_{v_k} (c^\alpha(\tau\otimes \sigma_{k-1}))\\
&= c^\alpha_{v_k}\partial (c^\alpha(\tau\otimes\sigma_{k-1})) + c^\alpha(\tau\otimes\sigma_{k-1}) + x_{v_k}^\alpha(c^\alpha(\tau\otimes\sigma_{k-1})) \\
&= c^\alpha_{v_k}(c^\alpha\partial (\tau\otimes\sigma_{k-1}) + a^\alpha(\tau\otimes\sigma_{k-1}) + b^\alpha(\tau\otimes\sigma_{k-1}))+ c^\alpha(\tau\otimes\sigma_{k-1})+ 0 \\
&= c^\alpha_{v_k}(c^\alpha(\partial\tau\otimes\sigma_{k-1})+c^\alpha(\tau\otimes\partial(\sigma_{k-1})) +c^\alpha_{v_k}(a^\alpha(\tau\otimes\sigma_{k-1})+ c^\alpha_{v_k}(b^\alpha(\tau\otimes\sigma_{k-1}))))\\
& \text{  \indent\indent    } + c^\alpha(\tau\otimes\sigma_{k-1})  \\
&= c^\alpha(\partial\tau\otimes \sigma ) + c^\alpha(\tau\otimes\partial\sigma)  + 2c^\alpha(\tau\otimes\sigma_{k-1}) + c^\alpha_{v_k}(a^\alpha(\tau\otimes\sigma_{k-1})) + c^\alpha_{v_k}(b^\alpha(\tau\otimes\sigma_{k-1})) \\
&=  c^\alpha\partial(\tau\otimes\sigma) + 0 + 0 + 0 \\
&= c^\alpha\partial(\tau\otimes\sigma) + a^\alpha(\tau\otimes\sigma) + b^\alpha(\tau\otimes\sigma).
\end{align*}
 \end{proof}

\subsection{The Interleaving Tools}

 We will now introduce a general construction on chain maps which we will apply to the maps defined in the previous subsection to construct the interleaving. 
Consider simplicial complexes $X$, $Y$, and $Z$.
The Alexander-Whitney diagonal approximation chain map (see ~\cite{lecturenotes01}), $\Delta_*: C^{\CW}_*(X)\rightarrow C^{\CW}_*(X)\otimes C^{\CW}_*(X)$,  is defined on a $k$-simplex $\sigma$ by $\Delta_k(\sigma) := \sum_{i=0}^k \sigma_i\otimes \overline{\sigma}_i$.
We remind the reader that if $\sigma$ is an $n$-simplex, $\sigma_i := \sigma |_{[v_0,\ldots, v_i]}$ and $\overline{\sigma}_i :=\sigma |_{[v_i,\ldots, v_n]}$.
Given a chain map $f: C^{\CW}_*(X)\otimes C^{\CW}_*(Y) \rightarrow C^{\CW}_*(Z)$, the chain map $\widehat{f}: C^{\CW}_*(X)\otimes C^{\CW}_*(Y)\rightarrow C^{\CW}_*(Z)\otimes C^{\CW}_*(Y)$ is defined by $\widehat{f}:= (f\otimes \text{id}_Y)\circ (\text{id}_X\otimes \Delta_*)$.
We call $\widehat{f}$ the $\textbf{lift}$ of $f$ and it the composition of chain maps so it is a chain map.

In alignment with the definition of the lift we view $C^{\CW}_*(B^\alpha)$ as a subgroup of $C^{\CW}_*(W^\alpha)\otimes C^{\CW}_*(|N^\alpha|)$ as the chain complex  $C^{\CW}_*(W^\alpha)\otimes C^{\CW}_*(|N^\alpha|)$ is naturally isomorphic to $C^{\CW}_*(W^\alpha\times |N^\alpha|)$.
We must still check that the maps we are lifting are well-defined, i.e. their images are properly contained in $C^{\CW}_*(B^{\alpha+t})$.

\begin{figure}
  \begin{center}
    \includegraphics[width=0.48\textwidth]{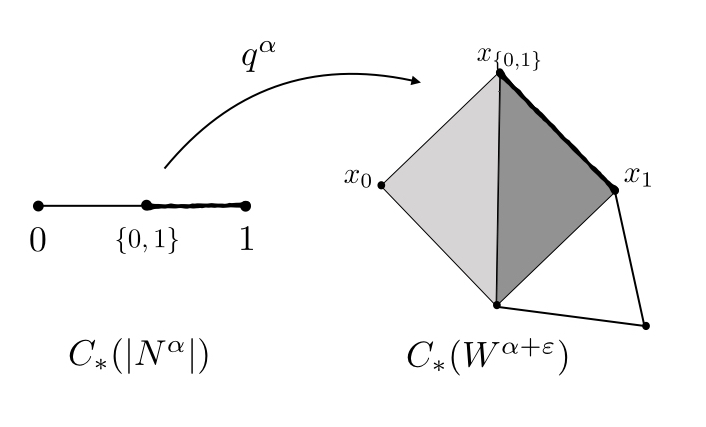} \includegraphics[width=0.48\textwidth]{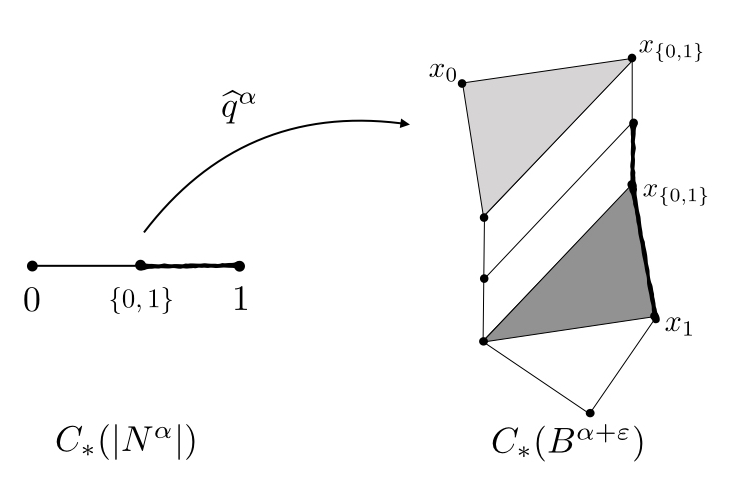}
    \caption{(R) $q^\alpha$ mapping the realization of a $1$-simplex of $|N^\alpha|$.
    (L) The $1$-simplex being mapped by $\widehat{q}^\alpha$ to a $1$-chain of $C_*(B^{\alpha+\e})$.
    Viewing $B^{\alpha+\e}$ vertically, $\widehat{q}^\alpha$ projects onto $1$-simplex $q^\alpha$.}
  \end{center}
 \end{figure} 

The lift of $q^\alpha$ is $\widehat{q}^\alpha(\sigma):= \sum_{i=0}^k q^\alpha(\sigma_i)\otimes\overline{\sigma}_i$.
To see that $\widehat{q}^\alpha$ is well-defined, for a basis $0$-chain $\sigma=|v_0| \in C_0(|N^\alpha|)$, note that $\widehat{q}^\alpha(\sigma)= x_{v_0}\otimes v_0\in C_0(U^{\alpha+\e}_{v_0}\times |v_0|)$, so $\widehat{q}^\alpha(\sigma)\in C_0(B^{\alpha+t})$.
For $k\leq K$, for a basis $k$-chain $\sigma = |v_0\rightarrow \ldots \rightarrow v_k|$ we have $\widehat{q}^\alpha(\sigma) = \sum_{i=0}^k  c^\alpha(x_{v_0}\otimes (\sigma_i\setminus v_0)) \otimes\overline{\sigma}_i$.
For each $i$, we have that $c^\alpha(x_{v_0}\otimes (\sigma_i\setminus v_0))\subseteq U_{v_i}^{\alpha+i\e}$ so  $c^\alpha(x_{v_0}\otimes (\sigma_i\setminus v_0))\times |\overline{\sigma}_i|\subset U_{v_i}^{\alpha+i\e}\times |\overline{\sigma}_i|\subset B^{\alpha+t}$.
We then have that for all $i$ each chain of the form $c^\alpha(x_{v_0}\otimes (\sigma_i\setminus v_0))\otimes \overline{\sigma}_i$ is in $C^{\CW}_k(B^{\alpha+t})$ so their sum is as well so $\widehat{q}^\alpha$ is well-defined.
The maps $q^\alpha$ and $\widehat{q^\alpha}$ define a way to map the nerve filtration into the simplicial union filtration and the subdivided blow-up complex filtration respectively, albeit at a further scale.

The rest of the lifted maps are well-defined as one can easily check that $\widehat{a}^\alpha = \widehat{q}^\alpha p^\alpha$, $p^{\alpha+t} \widehat{q}^\alpha= i_\mathcal{N}^{\alpha,\alpha+t}$, and $\widehat{i_\W^{\alpha,\alpha+t}{b}^\alpha} =  i^{\alpha,\alpha+t}_\mathcal{B}$.
For example, the latter equality holds for a basis cellular chain $\tau\otimes\sigma\in C^{\CW}_*(B^\alpha)$,
\[
\widehat{i_\W^{\alpha,\alpha+t}{b}^\alpha}(\tau\otimes\sigma)=\sum_{i=0}^k b(\tau\otimes\sigma_i)\otimes\overline{\sigma}_i=\tau\otimes\sigma = i^{\alpha,\alpha+t}_\mathcal{B}(\tau\otimes\sigma).
\]

 These lifted maps yield Diagram~\ref{eq:complete_lift_dgm} constructed from Diagram~\ref{eq:complete_dgm}.
         \begin{equation}\label{eq:complete_lift_dgm}
          \xymatrix{
		 C^{\CW}_{k}(W^\alpha)~\ar[r]^{i_\mathcal{W}^{\alpha,\alpha+t}} & C^{\CW}_k(W^{\alpha+t}) \\
		 C^{\CW}_{k}(B^\alpha)~\ar[r]^{\widehat{i_\W^{\alpha,\alpha+t}{b}^\alpha} }\ar[u]^{b^\alpha} \ar[d]_{p^\alpha} & C^{\CW}_k(B^{\alpha+t})~\ar[d]^{p^{\alpha+t}} \ar[u]_{b^{\alpha+t}}\\
		 C^{\CW}_{k}(|N^\alpha|)~\ar[ru]^{\widehat{q}^\alpha} \ar[r]^{i_\mathcal{N}^{\alpha,\alpha+t}} & C^{\CW}_k(|N^{\alpha+t}|) \\
          }
      \end{equation}

The following lemma, stated in terms of general chain complexes, proves that chain homotopies are preserved by liftings.
This will then be applied to the chain homotopy relationship proven in  Lemma~\ref{lem:c_homotopy} to achieve the chain complex interleaving.
\begin{lemma}\label{lem:lift_chain_hom}
If chain maps $f,g: C^{\CW}_*(X)\otimes C^{\CW}_*(Y) \rightarrow C^{\CW}_*(Z)$ are chain homotopic, then $\widehat{f}$ and $\widehat{g}$ are chain homotopic, where $\widehat{f}$ and $\widehat{g}$ are the lifts of $f$ and $g$ respectively.
\end{lemma}

\begin{proof}
Since $f$ and $g$ are chain homotopic, then by definition there exists  a chain homotopy $c: C^{\CW}_*(X)\otimes C^{\CW}_*(Y)\rightarrow C^{\CW}_{*+1}(Z)$.
We have that $c\otimes \id_Y$ is a chain homotopy between $f\otimes \id_Y$ and $g\otimes \id_Y$.
One can also prove thatt $\widehat{f}= (f\otimes\id_Y)\circ (\id_X\otimes \Delta_*)$ and $\widehat{g}= (g\otimes\id_Y)\circ (\id_X\otimes \Delta_*)$ are chain homotopic via $\widehat{c}:= (c\otimes \id_Y)\circ (\id_X \otimes \Delta_*)$ as if two maps are chain homotopic via $c$, then their compositions with a chain map $h$ are chain homotopic via $c\circ h$.
\end{proof}
      
  We now conclude with our main result, the Generalized Persistent Nerve Theorem.

        \begin{theorem}\label{thm:gpnl2}(Generalized Persistent Nerve Theorem)
        
         Given a finite collection of finite simplicial filtrations $\U = \{U_0,\ldots,U_n\}$, where  $U_i:=(U_i^\alpha)_{\alpha\geq 0}$ and all $U_i^\alpha$ are subcomplexes of a sufficiently large simplicial complex,  if $\U$ is an $\e$-good cover filtration of $\W := (\bigcup_{i=0}^n U_i^\alpha)_{\alpha\geq 0}$, then for all $K\geq 0$,
      \[
        \dist_B(\text{Dgm}_K(\W), \text{Dgm}_K(\nerve \U)) \le (K+1)\e.
      \]
    \end{theorem}

\begin{proof}
Recall that there are natural isomorphisms $H^{\CW}_*(B^\alpha)\cong H^{\CW}_*(W^\alpha)$ for all scales $\alpha\geq 0$ as $b^\alpha$ is a homotopy equivalence for finite simplicial covers.
By Theorem~\ref{thm:persistent_lemma}, $\text{Dgm}(\B)=\text{Dgm}(\W)$.

Fixing $K\geq 0$, we now focus our attention on constructing a $(K+1)\e$-interleaving between $H_K^{\CW}(\B)$ and $H_K^{\CW}(\W)$.
Define $t:=(K+1)\e$.
Consider the two collections of chain maps $p := (p^{\alpha+t} i_{\B}^{\alpha,\alpha+t})_{\alpha\geq 0}$ and $\widehat{q} := (\widehat{q}^\alpha)_{\alpha\geq 0}$ defined in dimension $K$.

By applying Lemma~\ref{lem:lift_chain_hom} to Lemma~\ref{lem:c_homotopy} and recalling that $\widehat{a}^\alpha = \widehat{q}^\alpha p^\alpha$ and $\widehat{i_\B^{\alpha,\alpha+t}{b}^\alpha} =  i^{\alpha,\alpha+t}_\mathcal{B}$,  we know that for all $\alpha\geq 0$, $\widehat{a}^\alpha \simeq i^{\alpha,\alpha+t}_\mathcal{B}$.
Then, for arbitrary $\alpha\geq 0$, we have that $\widehat{q}^{\alpha+t} p^{\alpha+t} i_{\B}^{\alpha,\alpha+t} \simeq i_{\B}^{\alpha+t,\alpha+2t}i_{\B}^{\alpha,\alpha+t}= i_{\B}^{\alpha,\alpha+2t}$.
We also have that $p^{\alpha+2t}i_{\B}^{\alpha+t,\alpha+2t}\widehat{q}^\alpha = i_{\N}^{\alpha,\alpha+2t}$, as $p^{\alpha+2t}i_{\B}^{\alpha+t,\alpha+2t} = i_{\mathcal{N}}^{\alpha+t,\alpha+2t} p^{\alpha+t}$ and $p^{\alpha+t}\widehat{q}^{\alpha}=i_{\N}^{\alpha,\alpha+t}$.

Applying $K$-dimensional cellular homology to $p$ and $\widehat{q}$ we have $H_K^{\CW}(\widehat{q}^{\alpha+t}) H_k^{\CW}( p^{\alpha+t} i_{\B}^{\alpha,\alpha+t})= H_K^{\CW}(i_{\B}^{\alpha,\alpha+2t})$ and $H_K^{\CW}(p^{\alpha+2t}i_{\B}^{\alpha+t,\alpha+2t})H_K^{\CW}(\widehat{q}^\alpha)=H_K^{\CW}( i_{\N}^{\alpha,\alpha+2t})$ for all $\alpha\geq 0$, so the collection of maps $H_K^{\CW}(p)$ and $H_K^{\CW}(\widehat{q})$ form a $(K+1)\e$-interleaving between $H_K^{\CW}(\B)$ and $H_K^{\CW}(\N)$.

To complete the proof, $N^\alpha$ is the barycentric decomposition of $\nerve\U^\alpha$, so their geometric realizations are homeomorphic and thus by Theorem~\ref{thm:persistent_lemma}, $\text{Dgm}(\mathcal{N})=\text{Dgm}(\nerve\U)$.
Also we have that $H_K^{\CW}(\N)\cong H_K(\N)$ as cellular and simplicial homology agree for simplicial complexes.
By considering the isomorphisms maps corresponding to the previous equality and $\text{Dgm}(\B)=\text{Dgm}(\W)$, we can conclude there is $(K+1)\e$-interleaving between $H_K(\W)$ and $H_K(\nerve\U)$, so $\dist_B(\text{Dgm}_K(\W),\text{Dgm}_K(\nerve\U))\leq (K+1)\e$ by Theorem~\ref{thm:interleaving_bound}.
\end{proof}

From this theorem we can further conclude that $\dist_B(\text{Dgm}(\W), \text{Dgm}(\nerve \U)) \le (D+1)\e$, where $D$ is the maximum dimension of $\nerve\U$.
We also have the Persistent Nerve Lemma as a corollary when restricted to finite simplicial filtrations.

\begin{corollary}\label{cor:cor_persistent_nerve_lemma} (The Simplicial Persistent Nerve Lemma)

        Given a finite collection of finite simplicial filtrations $\U = \{U_0,\ldots,U_n\}$ where  $U_i:=(U_i^\alpha)_{\alpha\geq 0}$ and all $U_i^\alpha$ are subcomplexes of a sufficiently large simplicial complex, if $\U$ is a good cover filtration of the $\W = (\bigcup_{i=0}^n U_i^\alpha)_{\alpha\geq 0}$, then 
      \[
        \textrm{Dgm}(\W)=\textrm{Dgm}(\nerve \U).
      \]
      \end{corollary}
      \begin{proof}
      Since $\U$ is a good cover, it is also a $0$-good cover.
     By Theorem~\ref{thm:gpnl}, $\dist_B(\text{Dgm}(\W),\text{Dgm}(\nerve \U))=0$. As both persistence diagrams are finite, $\textrm{Dgm}(\nerve \U)=\textrm{Dgm}(\W)$.
	\end{proof}
\section{A Biased Estimator?}
\label{sec:biased}

Here we present a trick to tighten the approximation of $\text{Dgm}_K(\W)$ by a factor of $2$ given the value of $\e$ is known, where $t:=(K+1)\e$.
This is a result of our symmetric interleaving being constructed from an asymmetric interleaving.

Define $\widetilde{\N}:=(|\widetilde{N}^\alpha|)_{\alpha\geq 0}$, where $\widetilde{N}^\alpha:= N^{\alpha-\frac{t}{2}}$.
This is a shift of the time parameter of $\N=(|N^\alpha|)_{\alpha\geq 0}$ by $\frac{t}{2}$.
We have that $|N^\alpha|=|\widetilde{N}^{\alpha+\frac{t}{2}}|$, so this results in a symmetric $\frac{t}{2}$-interleaving consisting of the maps $p^\alpha: H^{\CW}_K(B^\alpha)\rightarrow H^{\CW}_K(\widetilde{N}^{\alpha+\frac{t}{2}})$ and $\widehat{q}^\alpha: H^{\CW}_K(\widetilde{N}^{\alpha+\frac{t}{2}}) \rightarrow H^{\CW}_K(B^{\alpha+t})$, as opposed to the similarly defined asymmetric interleaving constructed in the proof of The Generalized Persistent Nerve Theorem.
By Theorem~\ref{thm:interleaving_bound}, $\dist_B(\text{Dgm}_K(\widetilde{\N}), \text{Dgm}_K(\B))\leq \frac{K+1}{2}\e$.

There is an unknown shift of the persistence diagram based on the goodness of the cover, which will give a better approximation to the true persistence diagram.
A more loose bound on the goodness of the cover might imply that one ought to use a bigger shift, but it's not clear how to know what shift exactly that is.
Recently, with another coauthor, we presented an approach for computing bottleneck distances up to shifts~\cite{cavanna17computing}.

\bibliographystyle{unsrt} 
\bibliography{bibliography.bib}

\appendix
\section{Tight Bound Example}\label{sec:tight_bound}
In this section an example is presented that proves the tightness of the bottleneck distance bound of Theorem~\ref{thm:gpnl2} over all dimensions.
This is equivalent to the example in Subsection 9.3 of Govc and Skraba~\cite{govc17approximate} up to definitions and notation.
In particular a $1$-good cover $\U$ of a simplicial filtration $\W = (\bigcup_{i=0}^n U_i^\alpha)_{\alpha\geq 0}$ such that $\dist_B(\text{Dgm}_K(\nerve\U),\text{Dgm}_K(\W))= K+1$ for all $K\leq n$ is presented which can easily be adapted to an $\e$-good cover filtration for arbitrary $\e>0$ by appropriately scaling the filtration indices and the definition of each $U_i^\alpha$.

Consider the collection of $n+1$ $(n-1)$-simplices $t_i=|[n]\setminus \{i\}|$, where $[n]=\{0,1,\ldots, n\}$,  i.e.~the facets of the standard $n$-simplex.
Define the cover filtration as the collection of simplicial filtrations $\U=\{U_0,\ldots, U_n\}$ where
\[
U_i^{\alpha} :=\begin{cases}
 t_i & \text{ for } \alpha \in[0, n+1)\\
 t_i \cup (\bigcup\limits_{k=0}^{\lfloor \alpha \rfloor-n-1} t_k) & \text{ for } \alpha\in[n+1, 2n+2)\\
|[n]| & \text{ for } \alpha =  2n+2.\\
\end{cases}
\]

First we compute $H_*(\W)$.
For $0\leq\alpha< 2n+2$, $W^\alpha =\bigcup_{i=0}^n t_i=\partial |[n]|$, so $\rank~\tilde{H}_r(W^\alpha)=1$ if $r=n$ and $\rank~ \tilde{H}_r(W^\alpha)=0$ otherwise. 
For $\alpha =2n+2$,  $W^\alpha =  |[n]|$, which is contractible, so $\rank~\tilde{H}_r(W^\alpha)=0$ for all $r\geq 0$.
The ranks of these homology groups imply that there is one persistent $n$-dimensional homological feature from $\alpha=0$ to $\alpha=2n+2$, and none in any other dimension.

Next we compute $H_*(\nerve\U)$.
Note that any $k$-wise intersection of the faces of the standard $n$-simplex is non-empty iff $k\leq n$, so for $0\leq \alpha < n+1$, $\nerve U^\alpha$ is the abstract simplicial complex $2^{[n]}\setminus  [n] $.
Recall the notation $U_\sigma^\alpha:=\bigcap_{i\in\sigma} U_i^\alpha$.
For $n+1\leq \alpha< 2n+2$ and non-empty $\sigma\subseteq [n]$, we have that  $U_\sigma^\alpha = (\bigcap_{i\in\sigma} t_i) \cup (\bigcup_{k=0}^{\lfloor \alpha \rfloor - n-1} t_k)$ which is always non-empty, so $\nerve U^\alpha$ is the standard $n$-simplex.
The same is true for $\nerve U^{2n+2}$ as each cover element is identical.
The homology groups then, for all $\alpha\in [0,n+1)$, are $\tilde{H}_r(\nerve\mathcal{U}^\alpha)=1$ if $r=n$, and $0$ otherwise.
For all $\alpha\geq n+1$, for all $r\geq 0$,  $\tilde{H}_r(\nerve\mathcal{U}^\alpha)=0$, so there is one persistent $n$-dimensional homological feature from $\alpha = 0$ to $\alpha = n+1$.

By the above computations $\text{Dgm}_n(\W)$ consists of the diagonal and the point $(0,2n+2)\in\R^2$, while $\text{Dgm}_n(\nerve\U)$ consists of the diagonal and the point $(0,n+1)\in\R^2$.
Regardless of whether $(0,2n+2)$ is mapped to $(0,n+1)$ or the diagonal, the $\ell_\infty$ distance between the matched points is $n+1$, so $\dist_B(\text{Dgm}_{n}(\W),\text{Dgm}_{n}(\nerve\U))=n+1$.

Next we prove that $\U$ is a $1$-good cover --- in other words, for all $\sigma$, for all $\alpha$, we are to prove that $\tilde{H}_*(U_\sigma^\alpha \hookrightarrow U_\sigma^{\alpha+1})=0$.
Note that $\U^\alpha$ is a good cover of $W^\alpha$ for all $\alpha\in[0,n+1)$, as the intersection between $k$ faces of the standard $n$-simplex is a full $(n-k)$-simplex, which is contractible.
In order to prove that the rest of the cover filtration is $1$-good, we show that there are no scales $\alpha\geq n+1$ for which $\tilde{H}_*(U_\sigma^\alpha) \neq 0$ and $\tilde{H}_*(U_\sigma^{\alpha+1})\neq 0$.
To do so, we will check the non-trivial combinations of $\sigma\subseteq [n]$ and scales.
By the definition of the cover filtration's elements, it suffices to check only $U_\sigma^{n+1+j}$ for integral $j\in\{0,\ldots,n\}$.

\begin{lemma}\label{lem:Usigma_contractible}
For $\sigma\subseteq [n]$ and $m>j$ such that $m\not\in \sigma$, $U_\sigma^{n+1+j}$ is contractible to $|m|$ and $\tilde{H}_*(U_\sigma^{n+1+j})=0$.
\end{lemma}
\begin{proof}
Recall that $U_i^{n+1+j} = t_i\cup (\bigcup_{k=0}^j t_k)$ and $U_\sigma^{n+1+j} =\bigcap_{i\in\sigma} t_i \cup (\bigcup_{k=0}^j t_k)$.
Given that $m\not\in\sigma$, then $|m|\in t_i$ for all $i\in\sigma$, so $|m|\in\bigcap_{i\in\sigma} t_i$.
In addition, $|m|\in t_k$ for all $k\neq m$, so $|m|\in t_k$ for all $k\leq j$ as $m>j$, so $|m|$ is in every simplex in $U_\sigma^{n+1+j}$, and thus $U_\sigma^{n+1+j}$ contracts to $|m|$.
\end{proof}

\begin{lemma}\label{lem:single_extra}
For $p\leq j<n$ and $\sigma =([n]\setminus [j]) \cup \{p\}$, $\tilde{H}_*(U_\sigma^{n+1+j})=0$.
\end{lemma}

\begin{proof}
By assumption we know that $U_\sigma^{n+1+j}=( \bigcap_{i\in\sigma}t_i) \cup (\bigcup_{i=0}^j t_j)$, where $ \bigcap_{i\in\sigma}t_i = |[0,\ldots, \hat{p},\ldots, j]|$.
As $|[0,\ldots, \hat{p},\ldots, j]| \subseteq t_p$, $|[0,\ldots, \hat{p},\ldots, j]|\subset \cup_{i=0}^j t_j$, so we know $U_\sigma^{n+1+j} = \bigcup_{i=0}^j t_j$.
By the assumption that $j<n$, $U_\sigma^{n+1+j}$ is not the entire boundary of $|[n]|$ so $\tilde{H}_*(U_\sigma^{n+1+j})=0$.
\end{proof}

\begin{lemma}\label{lem:multi_extra}
Given non-empty $\sigma \supset [n]\setminus [j]$ for $0\leq j \leq n$, where $\#\{\sigma\cap [j]\}\geq 2$, $\tilde{H}_*(U_\sigma^{n+1+j}) =0$.
\end{lemma}

\begin{proof}
First we claim that for any $p\in \sigma \cap [j]$, $U_\sigma^{n+1+j}=U_{\sigma\setminus p}^{n+1+j}$.
By definition, $U_\sigma^{n+1+j}\subseteq U_{\sigma\setminus p}^{n+1+j}$.
Now given some $x\in U_{\sigma\setminus p} ^{n+1+j}$, $x\in \bigcap_{i\in \sigma\setminus p} t_i$ or $x\in \bigcup_{k=0}^j t_j$.
Clearly in the latter case, $x\in U_{\sigma}^j$.
In the former case, we have that $\bigcap_{i\in \sigma} t_i = |[0,\ldots, \hat{p}_1,\ldots, \hat{p}_l ,\ldots, j]|$ for some vertices $\hat{p}_1,\ldots, \hat{p}_l\in (\sigma\cap[j])\setminus \{p\}$, so $|[0,\ldots, \hat{p}_1,\ldots, \hat{p}_l ,\ldots, j]|\subseteq t_{p_l}\subseteq \cup_{k=0}^j t_k$, and thus $x\in U_\sigma ^j$,  giving us the desired equality.

By the above, the only  case to check is when $\#\{\sigma \cap [j]\}=1$.
Applying Lemma~\ref{lem:single_extra} in this case proves that $\tilde{H}_*(U_\sigma^{n+1+j})=0$.
\end{proof}

 Lemma~\ref{lem:Usigma_contractible}, \ref{lem:single_extra} and \ref{lem:multi_extra} check all non-trivial cases  so $U_\sigma^{n+1+j}$ has non-trivial homology only if $\sigma =[n]\setminus [j]$ for  $j\in \{0,\ldots, n\}$.
This fact along with the second half of the cover filtration being good proves that $\U$ is a $1$-good cover of $\W$.
 To complete the construction simply take the disjoint union of the cover filtrations $\U$ over all required dimensions along with the corresponding covered space filtration.

\section{The Homotopy Colimit}
\label{sec:hom_colimit}

Define $\Delta$ as a finite acyclic directed graph, or equivalently a strict poset.
It can be given the structure of an abstract simplicial complex by defining  $k$-simplices are as sequences $v_0\rightarrow \ldots \rightarrow v_k$, where $v_i \rightarrow v_{i+1}$ for all $i\in [k]$.

Consider a functor, known as a diagram, $\dgm: \Delta \rightarrow \textrm{Top}$, where $\textrm{Top}$ is the category of topological spaces.
The \textbf{homotopy colimit} of $\dgm$ is defined as

    \[
      \hocolim \dgm := \Big(\bigsqcup_{\substack{\Delta \ni \sigma = v_0\to\cdots\to v_k \\ k\geq 0}} ( \dgm(v_0)\times |\sigma| \Big)/\sim,
    \]
    
 where the disjoint union is taken over all simplices $\sigma\in\Delta$ and the homotopy colimit is given the quotient topology.
  The equivalence relation $\sim$, describing the gluing procedure along the boundaries of the $k$-simplices of $\Delta$, is defined as follows.
    Given a $k$-simplex, consider the subsimplices of codimension $1$, $\tau_i=v_0\rightarrow\ldots \rightarrow \hat{v}_i \rightarrow \ldots v_k$, for each $i$,  and their geometric realizations $|\tau_i|$.
    For all $i$, let $f_i: |\tau_i|\hookrightarrow |\sigma|$ be the $i$-th face map.
    For $i>0$ we have the following equivalences
    \[
    \dgm(v_0)\times |\tau_i| \ni (x,y)\sim (x, f_i(y))\in \dgm(v_0)\times |\sigma|,
    \]
    and for $i=0$ we have the following equivalences
    \[
     \dgm(v_1)\times|\tau_0|\ni (\dgm(v_0\rightarrow v_1)(x),y)\sim  (x, f_0(y)) \in \dgm(v_0)\times|\sigma| .
    \]
      Refer to p. 262 of Kozlov~\cite{koslov08} for more information regarding this general construction.
    
Consider a simplicial cover of finite simplicial complexes $\U=\{U_0,\ldots, U_n\}$ and the corresponding abstract simplicial complexes $N^\alpha$ as defined in Subsection~\ref{sec:nerve_diagram}, where there is a morphism $v\rightarrow v'$ iff $v'\subset v\subseteq [n]$ for non-empty $v'$.
For each $\alpha\geq 0$, define the nerve diagram as the functor $\dgm^\alpha : N^\alpha \rightarrow \textrm{Top}$ such that for each $0$-simplex $v \in N^\alpha$, $\dgm^\alpha(v):= \bigcap_{i \in v} U_i^\alpha$  and for each $1$-simplex $v\rightarrow v'$, $\dgm^\alpha(v\rightarrow v'):= U_{v'}^\alpha\hookrightarrow U_v^\alpha$.

Here we prove that the  barycentric decomposition of the blow-up complex $B^\alpha$ is the homotopy colimit of the nerve diagram $\dgm^\alpha$. 

\begin{proposition}\label{thm:hom_colimit}
For the nerve diagrams $\dgm^\alpha: N^\alpha \rightarrow \textrm{Top}$, we have the following homotopy colimits for each $\alpha\geq 0$.
     \[
      \hocolim \dgm^\alpha = \bigcup_{\substack{N^\alpha\ni \sigma = v_0\to\cdots\to v_k \\ k\geq 0}} U_{v_0}^\alpha\times |\sigma|.
    \]
    \end{proposition}
    \begin{proof}
The equivalence relation simplifies as follows.
First have $\U_{v_0}^\alpha \times |\tau_i| \ni (x,y)\sim (x, y)\in U_{v_0}^\alpha\times |\sigma|$, for all $(x,y)\in U_{v_0}^\alpha \times |\tau_i|$,  so for each boundary element $\tau_i$ for $i>0$, one glues $U_{v_0}^\alpha \times |\tau_i|$ to its image under the component wise inclusion map, yielding $ U_{v_0}^\alpha \times |\sigma| = ((U_{v_0}^\alpha\times |\tau_i|) \bigsqcup (U_{v_0}^\alpha\times |\sigma|))/ \sim $ for each $i>0$.
    
The gluing procedure for $\tau_0$ simplifies to $U_{v_0}^\alpha \times |\sigma|  \ni (x, y)\sim (x,y)\in U_{v_1}^\alpha \times|\tau_0|,$ for all $x\in U_{v_0}\subseteq U_{v_1}$ and $y\in |\tau_0|\subset |\sigma|$, so the gluing occurs exactly along the two sets' intersection $U_{v_0}^\alpha\times|\tau_0| = (U_{v_0}^\alpha \times |\sigma|) \bigcap  (U_{v_1}^\alpha \times|\tau_0|)$, so we have that $ ((U_{v_0}^\alpha \times |\sigma|)) \bigsqcup (U_{v_1}^\alpha \times|\tau_0|) / (U_{v_0}^\alpha\times|\tau_0|)=  (U_{v_0}^\alpha \times |\sigma|)\bigcup (U_{v_1}^\alpha \times|\tau_0|)$.
    
Since these two situations via  induction collectively determine the gluing procedures for all simplices of $N^\alpha$, the result follows.
    \end{proof}

\end{document}